\documentclass[11pt,a4paper,reqno]{amsart}

\usepackage[applemac]{inputenc}
\usepackage[T1]{fontenc}
\usepackage{amsmath}
\usepackage{amsthm}
\usepackage{amsfonts}
\usepackage{amssymb}
\usepackage{graphicx}
\usepackage{amsbsy}
\usepackage{mathrsfs}
\usepackage{bbm}
\usepackage{tikz}
\usepackage{mathtools}
\usepackage{color}
\newtheorem{theorem}{Theorem}[section]
\newtheorem{lemma}[theorem]{Lemma}

\newtheorem{proposition}[theorem]{Proposition}

\theoremstyle{remark}
\newtheorem{remark}[theorem]{\it \bf{Remark}\/}

\numberwithin{equation}{section}
\catcode`@=11
\def\section{\@startsection{section}{1}%
  \z@{1.5\linespacing\@plus\linespacing}{.5\linespacing}%
  {\normalfont\bfseries\large\centering}}
\catcode`@=12
\newcommand{\be}{\begin{equation}}
\newcommand{\ee}{\end{equation}}
\newcommand{\bea}{\begin{eqnarray}}
\newcommand{\eea}{\end{eqnarray}}
\newcommand{\bee}{\begin{eqnarray*}}
\newcommand{\eee}{\end{eqnarray*}}

\def\RR{\mathbb{R}}

\catcode`@=11
\def\supess{\mathop{\operator@font Sup\,ess}}
\catcode`@=12

\def\RR{\mathbb{R}}

\def\R2+{\RR ^2_+}

\def\lim{\mathop{\rm lim}}

\def\sup{\mathop{\rm sup}}

\def\log{{\rm log}}

\def\ba{\begin{array}}
\def\ea{\end{array}}

\title{Spectral projectors, resolvent, and Fourier restriction on the hyperbolic space}

\author[P. Germain and T. L\'eger]{Pierre Germain and Tristan L\'eger}

\address{Courant Institute of Mathematical Sciences, New York University,
251 Mercer Street, New York, NY 10012, USA}
\email{pgermain@cims.nyu.edu}%

\address{Princeton University,  Mathematics  Department,  Fine Hall,Washington Road,  Princeton,  NJ 08544-1000,  USA}
\email{tleger@princeton.edu}%
\begin{document}

\begin{abstract} We develop a unified approach to proving $L^p-L^q$ boundedness of spectral projectors,  the resolvent of the Laplace-Beltrami operator and its derivative on $\mathbb{H}^d$. In the case of spectral projectors, and when $p$ and $q$ are in duality, the dependence of the implicit constant on $p$ is shown to be sharp. We also give partial results on the question of $L^p-L^q$ boundedness of the Fourier extension operator. As an application, we prove smoothing estimates for the free Schr\"{o}dinger equation on $\mathbb{H}^d$ and a limiting absorption principle for the electromagnetic Schr\"{o}dinger equation with small potentials. 
\end{abstract}

\maketitle

\tableofcontents

\section{Introduction}

\subsection{Fourier restriction, extension, and spectral projectors on the hyperbolic space} The aim of this section is to give a very succinct definition of the objects of interest in this paper. A more thorough presentation can be found in Section~\ref{HAHS}.

We adopt the hyperboloid model for the hyperbolic space: denoting the Minkowski metric on $\mathbb{R}^{d+1}$ by 
$$
[x,y] = x^0 y^0 - x^1 y^1 - \dots - x^d y^d, \qquad x,y \in \mathbb{R}^{d+1},
$$
we let
$$
\mathbb{H}^d = \{ x\in \mathbb{R}^{d+1}, \; [x,x] = 1, \; x^0>0 \},
$$
and endow this manifold with the metric induced by Minkowski's metric. The Laplace-Beltrami operator on $\mathbb{H}^d$ is denoted $\Delta_{\mathbb{H}^d}$. It has spectrum $(-\infty,-\rho^2]$, with
$$
\rho = \frac{d-1}{2}.
$$
The Helgason Fourier transform is given by
$$
\widetilde{f}(\lambda,\omega) = \int_{\mathbb{H}^d} f(x) h_{\lambda,\omega}(x) \,dx, \qquad h_{\lambda, \omega}(x) = [x \,,\, (1,\omega)]^{i \lambda - \rho}, \qquad (\lambda,\omega) \in \mathbb{R}_+ \times \mathbb{S}^{d-1}.
$$
By analogy with the Euclidean case, we define the restriction Fourier operator to frequencies of size $\lambda >0$ by
$$
[R_\lambda f](\omega) = | \textbf{c}(\lambda)|^{-1} \widetilde{f}(\lambda,\omega), \qquad \omega \in \mathbb{S}^{d-1}
$$
(mapping functions on $\mathbb{H}^d$ to functions on $\mathbb{S}^{d-1}$), where $ \textbf{c} (\lambda)$ is the Harish-Chandra function. Its dual with respect to the $L^2$ scalar products on $\mathbb{H}^d$ and $\mathbb{S}^{d-1}$ is the extension Fourier operator (here $\omega_{d-1}$ denotes the measure of $\mathbb{S}^{d-1}$)
\begin{equation}
\label{defextension}
[E_\lambda f](x) = | \overline{\textbf{c} (\lambda)|}^{-1}  \frac{1}{\omega_{d-1}} \int_{\mathbb{S}^{d-1}} f(\omega) \overline{h_{\lambda,\omega} (x)}\,d\omega, \qquad x \in \mathbb{H}^d
\end{equation}
(mapping functions on $\mathbb{S}^{d-1}$ to functions on $\mathbb{H}^d$). The spectral projectors $P_\lambda$ for $\Delta_{\mathbb{H}^d}$ satisfy
$$
-\Delta_{\mathbb{H}^d} = \rho^2 + \int \lambda^2 P_\lambda \,d\lambda,
$$
and are given by
$$
P_{\lambda} = E_\lambda R_\lambda \quad \mbox{or equivalently} \quad P_{\lambda} f(x) = | \textbf{c} (\lambda)|^{-2} \frac{1}{\omega_{d-1}} \int_{\mathbb{S}^{d-1}} \widetilde{f}(\lambda,\omega)  \overline{h_{\lambda,\omega} (x)}\,d\omega.
$$
Finally, we denote
$$
D = \sqrt{-\Delta_{\mathbb{H}^d}-\rho^2}
$$
and $m(D)$ for the Fourier multiplier with symbol $m(\lambda)$.

Finally as a convention we use the notation ``$\lesssim$" to mean ``$\leqslant C$" for some numerical constant $C$ that does not depend on the parameters of the problem. When we write ``$\lesssim_{\alpha}$" for some parameter $\alpha,$ then the implicit constant is allowed to depend said parameter.

\subsection{Background} The question of the boundedness of the spectral projectors from $L^{p'} \to L^p$ is equivalent to that of the boundedness of the Fourier extension operator from $L^2 \to L^p$. This is true on the hyperbolic space, considered here, as well as on the Euclidean space, where this question was first considered. Optimal bounds were obtained by Tomas \cite{To} except for the endpoint case, which is due to Stein. 
These works of Tomas and Stein were the starting point of the Fourier restriction theory, which has flourished since, and proved to be related to a number of mathematical fields, from geometric combinatorics to number theory. We refer to the textbook by Demeter~\cite{Demeter} for an account of these developments.

The question of extending these ideas to more general Riemannian manifolds was first investigated by Sogge in the case of the sphere~\cite{So1} and for general compact manifolds~\cite{So2}. This ultimately led to Theorem 5.1.1 in his textbook~\cite{So3}, which is optimal for general manifolds of finite geometry. A difficult and interesting problem is to understand the relation between the geometry of the manifold and the boundedness properties of the spectral projectors. 

A closely related question is that of the boundedness of resolvent operators. A foundational paper in this direction is due to Kenig, Ruiz and Sogge~\cite{KRS}, who considered this question for second order, constant coefficient operators on the Euclidean space. Their results were later improved by Guti\'errez~\cite{G} in the case of the Laplacian. The case of general compact manifolds was also considered, see in particular Dos Santos Ferreira-Kenig-Salo~\cite{DSFKS} and Bourgain-Shao-Sogge-Yao~\cite{BSSY}.

Turning to $L^p$ harmonic analysis on the hyperbolic space, an important non-Euclidean feature is the Kunze-Stein phenomenon, see Cowling~\cite{Co} and Ionescu \cite{I1}. For instance, it plays a crucial role in the derivation of dispersive estimates on $\mathbb{H}^d$ (Ionescu-Staffilani \cite{IS}, Anker-Pierfelice \cite{AP}). We will heavily rely on it in the present paper, especially to deal with low frequencies.
Another research line on $\mathbb{H}^d$ has been the $L^p$ boundedness of Fourier multipliers, for which we refer to Clerc-Stein~\cite{CS}, Taylor~\cite{Taylor} and Anker~\cite{Anker} (see also \cite{I3,I4,MV} for refinements). The question of the boundedness of  spectral projectors was considered by Huang-Sogge \cite{HS} and Chen-Hassell \cite{CH}, who were able to identify Lebesgue spaces over which the spectral projector is bounded. In the present article we go further, obtaining and proving sharpness of the  dependence of the bounds on the Lebesgue exponents. 

\subsection{Obtained results: boundedness of spectral projectors and resolvent}
As mentioned above, this problem has already been considered in the literature;
upper bounds for the $L^{p'}-L^p$ operator norm of spectral projectors were proved in \cite{HS} for the same exponents as in the euclidean case, namely $2 < p \leqslant p_{ST}:=\frac{2(d+1)}{d-1}$. This result was improved in \cite{CH} where the range of exponents is extended to all $p>2$. This work also contains $L^p-L^q$ bounds for the resolvent operator, for the same range of exponents as in the euclidean case (namely $\frac{1}{p}-\frac{1}{q} = \frac{2}{d}$ and $\min \bigg(\big \vert \frac{1}{p}-\frac{1}{2} \big \vert , \big \vert \frac{1}{q} - \frac{1}{2} \big \vert \bigg) > \frac{1}{2d} $ see Kenig-Ruiz-Sogge \cite{KRS}).  The question of boundedness of the resolvent specifically on hyperbolic spaces was first considered in \cite{LR}. On more general noncompact symmetric spaces see \cite{AL}.In the special case of dimension 2, estimates for the resolvent and derivative of the resolvent were obtained in \cite{L} (it was originally motivated by the question of asymptotic stability of harmonic maps for wave maps $\mathbb{R} \times \mathbb{H}^2 \rightarrow \mathbb{H}^2$).
 
In the present paper we improve on these results in several respects.  First regarding $L^{p'}-L^p$ boundedness of the restriction operator $P_{\lambda}$, we refine the results of Chen-Hassell \cite{CH} by keeping track of the dependence of the implicit constant on $p$. We also show that the dependence on $\lambda,p$ obtained is optimal by providing examples that saturate these bounds. More precisely we adapt the radial and Knapp examples to the hyperbolic setting. 
We also extend the upper bound to the case where the exponents are not in duality, that is we prove $L^s-L^q$ boundedness for $s \in [1,2), q \in (2,\infty].$
Regarding resolvent estimates, we show boundedness for a wider range of exponents than was previously known in the case of the hyperbolic space. We also keep track of the dependence on the spectral parameter in these estimates.

Finally we give simple applications of these results, and obtain smoothing estimates for the free Schr\"{o}dinger equation, as well as a limiting absorption principle for the electromagnetic Schr\"{o}dinger equation on $\mathbb{H}^d.$

We develop a novel approach to these questions centered around spectral projectors. Indeed it has long been noticed that these objects are ubiquitous in Harmonic Analysis, see \cite{St1,St2} for example. Our method allows for a unified treatment of restriction and resolvent estimates. More precisely, we express the resolvent in terms of spectral projectors near its singularity. This reduces the problem to proving estimates the restriction operator.

As in the classical proof of the Stein-Tomas theorem in the euclidean case, we rely on a dyadic decomposition of the kernel of the restriction operator. A notable difference is that we localize in frequency rather than physical space. This allows us to straightforwardly deduce estimates for the derivative of the resolvent as well.

Note that this approach is not specific to the hyperbolic case, and can be used in the euclidean case as well. 

\begin{theorem}[Boundedness of the spectral projector] \label{thm-projector}
For $\Lambda>1$ and $2  <  p \leq \infty$, 
$$
\| P_{\Lambda} \|_{L^{p'} \to L^p}  \lesssim 
\left\{ \begin{array}{ll}
\Lambda^{d-1-\frac{2d}{p}} & \mbox{if $p \geq p_{ST} = \frac{2(d+1)}{d-1}$} \\
\left[ (p-2)^{-1} + 1 \right] \Lambda^{(d-1) \left( \frac{1}{2} -\frac{1}{p} \right) } & \mbox{if $2 < p \leq p_{ST}$}
\end{array} \right.
$$
Furthermore, these bounds are optimal. 

Moving away from the line of duality, we now choose the source Lebesgue exponent $s \in [1,2)$, and the target Lebesgue exponent $q \in (2,\infty]$. Then, there holds
$$
\| P_\Lambda \|_{L^{s} \to L^q} \lesssim
\left[ \Lambda^{\frac{d-1}{2}-\frac{d}{q}} + \left[ (q-2)^{-1/2} + 1 \right] 
\Lambda^{\frac{(d-1)}{2} \left( \frac{1}{2} -\frac{1}{q} \right) }\right]  \left[ \Lambda^{\frac{d}{s}-\frac{d+1}{2}} + \left[ (2-s)^{-1/2} + 1 \right] \Lambda^{\frac{(d-1)}{2} \left( \frac{1}{s} -\frac{1}{2} \right) }\right].
$$
\end{theorem}

Some estimates on the resolvent were also proved in general dimension in \cite{HS}. In the next theorem we obtain estimates on $(D^2-\tau - i \varepsilon)^{-1}$ for a wider range of exponents, while keeping track of the dependence of the constant on $\tau.$

\begin{theorem}[Boundedness of the resolvent] \label{thm-resolvent}
Let $z:= \tau + i \varepsilon \in \mathbb{C}$, with $\tau>1$ and $\epsilon>0$. For $2\leq p < \infty$,
\begin{align*}
\Vert (D^2-z)^{-1}  \Vert_{L^{p'} \rightarrow L^{p}} \lesssim_{p} \left\{ \begin{array}{ll}
\tau^{\frac{d}{2}(1-\frac{2}{p})-1} & \mbox{if $ p_{ST} < p \leqslant \frac{2d}{d-2} $} \\
\tau^{\frac{d-1}{2}(\frac{1}{2}-\frac{1}{p})-\frac{1}{2}} & \mbox{if $2 < p \leq p_{ST}:= \frac{2(d+1)}{d-1}$}
\end{array} \right.
\end{align*}
Note that the implicit constant does not depend on $\varepsilon,$ but may depend on $p.$

Moving away from the line of duality, it enjoys the bounds, for $1\leq s \leq q \leq \infty$, for regions I, II, III, IV defined in Figure~\ref{figure1}, and for implicit constants independent of $\epsilon$, but dependent on $s,q:$
\begin{itemize}
\item {Region $I:$} $\Vert (D^2-z)^{-1} \Vert_{L^s \rightarrow L^q} \lesssim_{ q,s} \tau^{\frac{\rho}{2}\big(\frac{1}{s}-\frac{1}{q} \big)-\frac{1}{2}},$ 
\item {Region $II:$} $\Vert (D^2-z)^{-1} \Vert_{L^s \rightarrow L^q} \lesssim_{ q,s}  \tau^{\frac{\rho}{2}\big(\frac{1}{s} - \frac{1}{q} \big) + \frac{d}{2}\big(\frac{1}{2} - \frac{1}{q} \big)-\frac{3}{4}},$
\item {Region $III:$} $\Vert (D^2-z)^{-1} \Vert_{L^s \rightarrow L^q} \lesssim_{q,s} \tau^{\frac{d}{2}\big( \frac{1}{s} - \frac{1}{q} \big)-1} .$ 
\item {Region $IV:$} $\Vert (D^2-z)^{-1} \Vert_{L^s \rightarrow L^q} \lesssim_{ q,s} \tau^{\frac{d}{2}\big(\frac{1}{s} - \frac{1}{2} \big)-\frac{3}{4}} .$
\end{itemize}
\end{theorem}
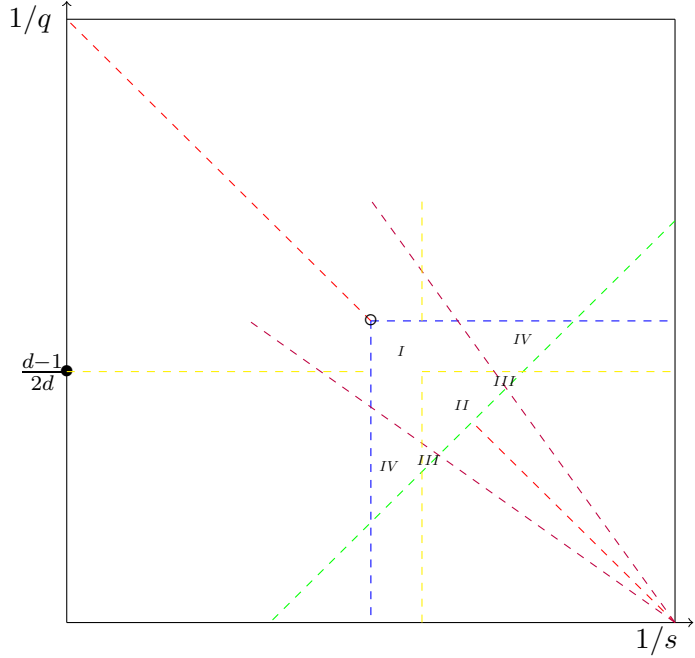
\begin{figure}
\centering
\begin{tikzpicture}[scale=0.8]
\draw[->] (-5,0) -- (5.3,0) ;
\draw[->] (-5,0) -- (-5,10.3) ;
\draw (4.7,-0.3) node{$1/s$} ;
\draw (-5.6,10) node{$1/q$} ; 
\draw (5,0) -- (5,10) ;
\draw (-5,10) -- (5,10); 
\draw[red,dashed] (5,0)--(1.66,3.33) ;
\draw[red,dashed] (0,5)--(-5,10) ;
\draw[blue,dashed] (0,5)--(5,5) ;
\draw[blue,dashed] (0,5)--(0,0) ;
\draw[green,dashed] (5,6.66)--(-1.66,0) ;
\draw (-5,4.16) node{$\bullet$} ;
\draw (-5.4,4.16) node{$\frac{d-1}{2d}$} ;
\draw[yellow,dashed] (-5,4.16)--(0,4.16) ;
\draw[yellow,dashed] (0.84,4.16)--(5,4.16) ;
\draw[purple,dashed] (5,0)--(0,7) ;
\draw[purple,dashed] (5,0)--(-2,5);
\draw[yellow,dashed] (0.84,0)--(0.84,4.16);
\draw[yellow,dashed] (0.84,5)--(0.84,7);
\draw (0,5) node{$\circ$};
\draw (0.5,4.5) node[scale=0.5]{$I$};
\draw (2.5,4.7) node[scale=0.5]{$IV$};
\draw (2.2,4) node[scale=0.5]{$III$};
\draw (1.5,3.6) node[scale=0.5]{$II$};
\draw (0.95,2.7) node[scale=0.5]{$III$};
\draw (0.3,2.6) node[scale=0.5]{$IV$};
\end{tikzpicture}
\caption{\label{figure1}
Boundedness of $(D^2-\tau-i \varepsilon)^{-1}$. The equations of the lines in this figure are as follows (using the convention that a line and its symmetric with respect to the diagonal $1/q+1/s=1$ have the same color). 
Yellow line: $\frac{1}{q} = \frac{d-1}{2d}$;  Green line: $\frac{1}{q} - \frac{1}{s} = \frac{2}{d}$; Purple line: $\frac{d-1}{d+1} \frac{1}{q} + \frac{1}{s} = 1. $ This figure corresponds to the case $d \geq 3$, lines are arranged slightly differently for $d=2$.}
\end{figure}

Finally, we give applications of these estimates in Section \ref{applications} to smoothing estimates for the Schr\"odinger equation, and to resolvent estimates for electromagnetic perturbations of the Laplacian. To the authors' knowledge these are new results, since prior work focused on weighted $L^2$ spaces \cite{K}. 

\subsection{Obtained results: the Fourier extension problem}
Since $P_\lambda = E_\lambda R_\lambda = E_\lambda (E_\lambda)^*$, the classical $TT^*$ argument combined with Theorem~\ref{thm-projector} gives sharp estimates for the $L^2 \to L^p$ operator norm of $E_\lambda$. It also immediately implies that the $L^p \to L^q$ operator norm of $E_\lambda$ is finite for any $p \geq 2$ and $q>2$. This observation leads to asking what the operator norm of $E_\lambda$ is in that range; and whether $E_\lambda$ might be bounded outside of it. The following proposition provides part of the answer.

\begin{proposition}[Lower bounds for the operator norm of the Fourier extension operator]
\label{proplb}
Recall that $E_\lambda$ is defined in~\eqref{defextension}.
\begin{itemize}
\item[(i)] The operator $E_\lambda$ is not bounded from $L^p$ to $L^q$ if $p<2$ or $q \leq 2$.
\item [(ii)] If $p \geq 2$ and $q>2$, the operator norm of $E_\lambda$ from $L^p \to L^q$ satisfies
$$
\| E_\lambda \|_{L^p \to L^q} \gtrsim  (q-2)^{-1/2} \lambda^{\frac{\rho}{p} - \frac{\rho}{q}} + \lambda^{\rho - \frac{d}{q}}.
$$
\end{itemize}
\end{proposition}

This proposition is proved by adapting the Knapp and radial examples to the hyperbolic setting, and by studying the effect of isometries of the hyperbolic space on the Fourier transform.
It suggests the conjecture
$$
\boxed{\mbox{if $p \geq 2, q>2$}, \quad \| E_\lambda \|_{L^p \to L^q} \lesssim  (q-2)^{-1/2} \lambda^{\frac{\rho}{p} - \frac{\rho}{q}} + \lambda^{\rho-\frac{d}{q}}}
$$
which can also be formulated as
$$
\| E_\lambda \|_{L^p \to L^q} \lesssim_{q}
\left\{
\begin{array}{ll}
\lambda^{\frac{\rho}{p} - \frac{\rho}{q}} & \mbox{if $p\geq 2$, $q > 2$, $\displaystyle \frac{d+1}{d-1} \frac 1 q \geq 1 - \frac{1}{p}$} \\
\lambda^{\rho- \frac d q} & \mbox{if $p \geq 2$, $q > 2$, $\displaystyle \frac{d+1}{d-1} \frac 1 q \leq 1 - \frac{1}{p}$} 
\end{array}
\right. .
$$
\\
\\
\textbf{Acknowledgements.} The authors would like to thank Jean-Philippe Anker for correcting a mistake in the complex interpolation argument in an earlier version of the paper. 
They would also like to thank Alexandru Ionescu for very helpful discussions, as well as the anonymous referee for their careful reading of the paper which greatly improved its readability. We are also grateful they pointed out important references.

P. Germain is supported by the NSF grant DMS-1501019, by the Simons collaborative grant on weak
turbulence, and by the Center for Stability, Instability and Turbulence (NYUAD).

T. L\'{e}ger is supported by the Simons collaborative grant on weak turbulence.

\section{Harmonic analysis on the hyperbolic space}
\label{HAHS}
In this section, we recall basic facts about the space $\mathbb{H}^d$ as well as its Fourier theory. They can be found in classical references such as the textbook by Helgason~\cite{Helgason}, the review by Bray~\cite{Bray}, and the nice and concise presentation in Ionescu-Staffilani~\cite{IS}.

\subsection{Analysis on $\mathbb{H}^d$}
\subsubsection{Hyperboloid model}
The Minkowski metric on $\mathbb{R}^{d+1}$ is given by
$$
[x,y] = x^0 y^0 - x^1 y^1 - \dots - x^d y^d \qquad \mbox{if $x,y \in \mathbb{R}^{d+1}$}.
$$
We define $\mathbb{H}^d$ as the hyperboloid (or to be more precise, the upper sheet of the hyperboloid)
$$
\mathbb{H}^d = \{ x\in \mathbb{R}^{d+1}, \; [x,x] = 1, \; x^0>0 \}
$$
and equip this space with the Riemannian metric induced by the Minkowski metric. This Riemannian metric induces in turn a measure, which will be denoted $dx$. We will further distinguish the point $\mathbf{0}=(1,0,\dots,0)$.

The group of isometries of the Minkowski space leaving $\mathbb{H}^d$ invariant is $SO(d,1)$, which we denote $\mathbb{G}$; it naturally acts on $\mathbb{H}^d$. The isotropy group of $\mathbf{0}$ is naturally identified with isometries of the Euclidean space $SO(d)$, which we denote $\mathbb{K}$.

Normalizing the Haar measure on $\mathbb{G}$ so that $\displaystyle \int_{\mathbb{G}} f (g \cdot \mathbf 0) \,dg = \int_{\mathbb{H}^d} f(x)\,dx$, we can define convolution on $\mathbb{H}^d$ through
$$
f * K (x) = \int_{\mathbb{G}} f(g \cdot \mathbf 0) K(g^{-1} \cdot x) \,dg.
$$
Note that in the case where $K$ is radial, we have
\begin{align} \label{radial-conv}
f * K (x) = \int_{\mathbb{H}^d} f(x') K(d(x,x')) dx' ,
\end{align}
where $d(x,x')$ denotes the geodesic distance between $x$ and $x'.$ \\
We will also rely on a non-Euclidean feature of $\mathbb{H}^d$ to deal with low frequencies, namely the Kunze-Stein phenomenon:
\begin{lemma}[\cite{APV}, Lemma 5.1] \label{K-S}
For every radial measurable function $\kappa$ on $\mathbb{H}^d$, every $2 \leqslant q, \widetilde{q} < \infty$ and $f \in L^{\widetilde{q}'}(\mathbb{H}^d),$
\begin{align*}
\Vert f * \kappa \Vert_{L^q} \lesssim_{q} \Vert f \Vert_{L^{\widetilde{q}'}}  \bigg( \int_0 ^{\infty}  (\sinh r)^{d-1} \vert \kappa (r) \vert^{Q} (1+r)^{\mu} e^{-\rho \mu r} dr \bigg)^{1/Q},
\end{align*}
where 
\begin{align*}
\mu = \frac{2 \min \lbrace q , \widetilde{q} \rbrace}{q+\widetilde{q}}, \ Q = \frac{q \widetilde{q}}{q+\widetilde{q}}.
\end{align*}
\end{lemma} 

\subsubsection{Fourier Analysis} For $\omega \in \mathbb{S}^{d-1}$, let
$$
b(\omega) = (1,\omega) \in \mathbb{R}^{d+1}.
$$
The analog of plane waves is provided by
$$
h_{\lambda, \omega}(x) = [x \,,\, b(\omega)]^{i \lambda - \rho}, \qquad x \in \mathbb{H}^d
$$
(notice that $[x,b(\omega)] > 0$ for $x \in \mathbb{H}^d$). They satisfy
$$
\Delta_{\mathbb{H}^d} h_{\lambda, \omega}(x) = - (\lambda^2 + \rho^2) h_{\lambda, \omega}(x).
$$
The Helgason Fourier transform on $\mathbb{H}^d$ is defined as
$$
[\widetilde{\mathcal{F}} f] (\lambda, \omega) = \widetilde{f}(\lambda,\omega) = \int_{\mathbb{H}^d} f(x) h_{\lambda,\omega}(x) \,dx, \qquad (\lambda,\omega) \in \mathbb{R}_+ \times \mathbb{S}^{d-1}.
$$
The inverse Fourier transform is then given by 
$$
f(x) = \int_0^\infty \frac{1}{\omega_{d-1}} \int_{\mathbb{S}^{d-1}} \widetilde{f}(\lambda,\omega) \overline{h_{\lambda,\omega}(x)} |c(\lambda)|^{-2} \,d\lambda \,d\omega,
$$
for the Harish-Chandra function
$$
\textbf{c} (\lambda) = \frac{2^{2\rho-1} \Gamma(\rho+\frac{1}{2})}{\pi^{1/2}} \frac{\Gamma(i\lambda)}{\Gamma(\rho + i\lambda)}
$$
whose asymptotics are as follows \cite{T} 
\begin{align*}
\textbf{c}(\lambda)^{-1} &= \frac{2^{2\rho-1} \Gamma(\rho+\frac{1}{2})}{\pi^{1/2}} (i\lambda)^\rho + O(\lambda^{\rho -1}).
\end{align*}
We can deduce from the Fourier transform formula the expression of the spectral projectors for the Laplace-Beltrami operator ${\Delta}_{\mathbb{H}^d}$:
$$
- \Delta_{\mathbb{H}^d} - \rho^2 = \int \lambda^2 P_\lambda \,d\lambda, \quad \mbox{with} \quad P_\lambda f (x) = \frac{1}{\omega_{d-1}} \int_{\mathbb{S}^{d-1}} \widetilde{f}(\lambda,\omega) \overline{h_{\lambda,\omega}(x)} |c(\lambda)|^{-2}  \,d\omega.
$$
Writing $D = \sqrt{-{\Delta}_{\mathbb{H}^d} - \rho^2}$, radial Fourier multipliers are defined as follows:
$$
m(D) f = \widetilde{\mathcal{F}}^{-1} \left[ m(\lambda) [ \widetilde{\mathcal{F}} f] (\lambda,\omega) \right] = \int m(\lambda) P_\lambda f \,d\lambda.
$$

Finally, the analog of Plancherel's theorem holds: the Fourier transform is an isometry from $L^2(\mathbb{H}^d,dx)$ to $L^2(\mathbb{R}_+ \times \mathbb{S}^{d-1}, | \textbf{c} (\lambda)|^{-2}\,d\lambda \,\frac{d\omega}{\omega_{d-1}} )$.

\subsection{Coordinate systems on $\mathbb{H}^d$}
We will use two coordinate systems on $\mathbb{H}^d$, which we now define.

\subsubsection{Polar coordinates} \label{polar}
In this hyperboloid model for $\mathbb{H}^d$, we can adopt polar coordinates 
$$
x=(\cosh r, \, \sinh r \, \omega), \qquad (r,\omega) \in \mathbb{R}_+ \times \mathbb{S}^{d-1},
$$
where $r \geq 0$ is the geodesic distance to the origin $\mathbf{0}$. In these coordinates, the volume element becomes
$$
dx = (\sinh r)^{2\rho} \,dr \, d\omega.
$$
The spherical function is given by
$$
\Phi_\lambda(x) = \frac{1}{\omega_{d-1}} \int_{\mathbb{S}^{d-1}} h_{\lambda,\omega}(x)\,d\omega.
$$
It only depends on the distance $r$ of $x$ to the origin, and can be written
\begin{equation}
\label{sphericalfunction}
\Phi_\lambda(r) = \frac{2^{\rho-1} \Gamma(\rho + \frac{1}{2})}{\sqrt{\pi} \Gamma(\rho) (\sinh r)^{2\rho -1}}\int_{-r}^r e^{i\lambda s} ( \cosh r - \cosh s)^{\rho-1} \,ds.
\end{equation}

Radial functions on $\mathbb{H}^d$ are invariant by $\mathbb{K}$; in other words, they only depend on $r$. Therefore,
$$
\widetilde{f}(\lambda, \omega) = \widetilde{f}(\lambda) = \int_{\mathbb{H}^d} f(x) \Phi_{\lambda}(x)\,dx = \omega_{d-1} \int_0^\infty f(r) \Phi_{\lambda}(r) (\sinh r)^{2 \rho} \,dr.
$$

It is not the case in general that the Fourier transform on $\mathbb{H}^d$ exchanges multiplication and convolution; but it is true for radial functions. Namely, if $K$ is radial, then
$$
\widetilde{f * K}(\lambda,\omega) = \widetilde{f}(\lambda,\omega) \widetilde{K}(\lambda).
$$
The convolution kernel $K = \widetilde{ \mathcal{F}}^{-1} m(\lambda)$ associated to the even radial multiplier $m$ is given by the following formulas (see~\cite{Taylor0}, Chapter 8, Section 5):
\begin{itemize}
\item If $d$ is odd, 
\begin{equation}
\label{albatros1}
\displaystyle K(r) = \frac{1}{\sqrt{2\pi}} \left( \frac{-1}{2\pi} \frac{1}{\sinh r} \frac{\partial}{\partial r} \right)^\rho \widehat{m}(r).
\end{equation}
\item If $d$ is even,
\begin{equation}
\label{albatros2} 
\displaystyle K(r)  = \frac{1}{{\sqrt \pi}} \int_r^\infty \left( \frac{-1}{2\pi} \frac{1}{\sinh s} \frac{\partial}{\partial s} \right)^{d/2} \widehat{m}(s) (\cosh s - \cosh r)^{-1/2} \sinh s \,ds.
\end{equation}
\end{itemize}
Here $\widehat{m}$ denotes the flat Fourier transform of $m.$ 

These formulas are related to the Abel transform on hyperbolic spaces, first considered by Flensted-Jensen and Koornwinder, see \cite{Ko}.
We also record an easy technical lemma that will be used repeatedly in the rest of the paper. 
\begin{lemma}
We have the identity
\begin{align} \label{decompmder}
\bigg(\frac{-1}{2 \pi} \frac{1}{\sinh r} \frac{\partial}{\partial r} \bigg)^{M} \widehat{m}(r) = \sum_{l=0}^{M} F_{l,M}(r) \partial_r ^l \widehat{m} ,
\end{align}
where the functions $F_l$ are smooth and satisfy
\begin{align} \label{boundFl}
\bigg \vert \frac{d^{\alpha}}{dr^{\alpha}}  F_{l,M} \bigg \vert \lesssim (\sinh r)^{-2M + l - \alpha} e^{r(M-l+\alpha)},
\end{align}
where $0 \leqslant \alpha \leqslant M.$ 
\end{lemma}
\begin{proof}
Straightforward by induction on $M.$
\end{proof}

We will need asymptotics for the spherical function:
\begin{proposition} \label{spherical-asympt} 
Assume $\lambda > 1$.
\begin{itemize}
\item If $r < \frac{1}{\lambda}$, $\displaystyle \Phi_\lambda(r) = O(1)$ uniformly on $\lambda$. 
\item If $ \frac{1}{\lambda} < r < 1$, $\displaystyle 
\Phi_\lambda(r) = 2^\rho \Gamma(\rho + \frac{1}{2}) \frac{\cos ( |\lambda| r - \frac{\rho \pi}{2} ) }{\lambda^{\rho} (\sinh r)^\rho} + O\left( \frac{1}{\lambda^{\rho+1} (\sinh r)^{\rho} r}  \right).$
\item If $r > 1$, $\displaystyle 
\Phi_\lambda(r) = 2^\rho \Gamma(\rho + \frac{1}{2}) \frac{\cos ( |\lambda| r - \frac{\rho \pi}{2} ) }{\lambda^{\rho} (\sinh r)^\rho}+ O\left( \frac{1}{\lambda^{\rho+1} (\sinh r)^\rho} \right).$
\end{itemize}
\end{proposition}
\begin{proof} The first assertion follows immediately from~\eqref{sphericalfunction}; and the proof for the last two can be found in \cite{T}. \end{proof}

\subsubsection{Iwasawa coordinates} \label{gpe}
Another global system of coordinates is deduced from the Iwasawa decomposition $\mathbb{G} = \mathbb{N} \mathbb{A} \mathbb{K},$ where $\mathbb{A}$ is the subgroup of $\mathbb{G}$ made up of Lorentz boosts in the first variable 
\begin{align*}
\mathbb{A} = \left\{a_t = \left( \begin{array}{ccc}
 \cosh t & \sinh t & 0\\
\sinh t & \cosh t & 0\\
0 & 0 & I_{d-1} 
\end{array} \right), t \in \mathbb{R}
\right\}
\end{align*}
and
\begin{align*}
\mathbb{N} = \left\{ n_v = \left( \begin{array}{ccc}
1+\vert v \vert^2/2 & - \vert v \vert^2/2 & v^\top \\
\vert v \vert^2/2 & 1-\vert v \vert^2/2 & v^\top \\
v & -v & I_{d-1} 
\end{array} \right), v \in \mathbb{R}^{d-1}
\right\}.
\end{align*}
The coordinates $s$ and $v$ are then defined by $x = n_v a_s \cdot \mathbf 0$. 
In other words,
$$
x =  \left(\cosh s + e^{-s} \frac{\vert v \vert^2}{2},  \sinh s + e^{-s} \frac{\vert v \vert^2}{2}, e^{-s} v_1, ... ,e^{-s} v_{d-1} \right), \qquad (s,v) \in \mathbb{R}^d.
$$
Note that these coordinates are such that the orbits $\mathbb{N} a_s \cdot \mathbf{0}$ are horocycles. \\
The Riemannian metric becomes
$$
e^{-2s}( (dv_1)^2 + \dots + (dv_{d-1})^2) + (ds)^2,
$$
and at the North Pole $\omega_{NP} = (1,0,\dots 0) \in \mathbb{S}^{d-1}$, the function $h_{\lambda,\omega_{NP}}$ becomes
\begin{equation}
\label{mesangeboreale}
h_{\lambda,\omega_{NP}}(s,v) = e^{(\rho - i\lambda) s}.
\end{equation}

\section{Two examples} \label{lower-b}

\subsection{The radial example}
In the regime $p \geqslant p_{ST},$ the lower bound on $\Vert P_{\Lambda} \Vert_{L^{p'} \rightarrow L^p}$ is a direct consequence of the following lemma.

\begin{lemma}
\label{lemmaradial}
The spherical function $\Phi_\lambda$ satisfies
$$
\| \Phi_\Lambda \|_{L^p} \gtrsim  \Lambda^{-d}.
$$
\end{lemma}

\begin{proof} Using Proposition \ref{spherical-asympt} we can write that for $p>\frac{2d}{d-1}$ we have
$$
\| \Phi_\lambda \|_{L^p}^p \gtrsim \int_{1/\lambda}^{\infty} \frac{1}{\lambda^{\rho p} (\sinh r)^{\rho p}} (\sinh r)^{d-1} \,dr \sim \Lambda^{-d}.
$$
\end{proof}

\subsection{The Knapp example}

\begin{lemma} 
\label{lemmaknapp} For $\delta<1$, let $\varphi_\delta$ be the characteristic function of the set $\{ \omega \in \mathbb{S}^{d-1}, \,|\omega - NP| < \delta \}$, where $NP$ is the north pole $(1,0,\dots,0) \in \mathbb{S}^{d-1}$.
Assume that $\lambda \delta^2  \ll 1$. Adopting the Iwasawa coordinates from Section~\ref{gpe},
$$
[ E_\lambda \varphi_\delta ](x) \gtrsim \lambda^\rho \delta^{d-1} e^{\rho s} \qquad \mbox{if $-\infty < s < - \frac{1}{2} \log ( \lambda \delta^2)$ and $|v| \ll \frac{1}{\delta \lambda}$}.
$$
As a consequence,
$$
\| E_\lambda \varphi_\delta \|_{L^p} \gtrsim \frac{1}{(p-2)^{1/p}} \lambda^{\frac{\rho}{2} - \frac{\rho}{p}} \delta^\rho.
$$
\end{lemma}
\begin{proof} 
The sphere $\mathbb{S}^{d-1}$ has coordinates $(\omega_1, \dots, \omega_d)$ in $\mathbb{R}^d$. Close to the North Pole, we parameterize $\mathbb{S}^{d-1}$ by $\omega_2,\dots,\omega_d$. Note that $|\partial_{\omega_2} \omega_1| \sim |\omega_2|<\delta$ on the support of $\varphi_\delta$.

In Iwasawa coordinates,
$$
[x,b(\omega)] = \cosh s + e^{-s} \frac{|v|^2}{2} - \omega_1 \left( \sinh s + e^{-s} \frac{|v|^2}{2} \right) - \omega_2 e^{-s} v_1 - \dots - \omega_d e^{-s} v_{d-1},
$$
so that in particular $[x,b(\omega_{NP})] = e^{-s}$.

We want to find the set of $(s,v)$ such that, if $|\omega - NP| <\delta$,
\begin{enumerate}
\item  $\displaystyle
\left| \frac{\partial_\omega [x,b(\omega)]}{[x,b(\omega_{NP})]} \right| \ll \frac{1}{\delta}
$
\item and $\displaystyle
\left| \frac{\partial_\omega h_{\lambda,\omega}(x)}{h_{\lambda,\omega}(x)} \right| \ll \frac{1}{\delta}.
$
\end{enumerate}
By symmetry, it suffices to replace above $\partial_\omega$ by $\partial_{\omega_2}$.
To address point (1), we can bound
$$
\left| \frac{-(\partial_{\omega_2} \omega_1)(\sinh s + e^{-s} \frac{|v|^2}{2}) - e^{-s} v_1}{e^{-s}} \right| \lesssim \delta (1 + e^{2s}) + \delta |v|^2 + |v_1|.
$$
Recalling that $\lambda \delta^2 \ll 1$, this is $\ll \frac{1}{\delta}$ if 
\begin{equation}
\label{conditionsv}
s < -\log \delta - C \quad \mbox{and} \quad |v| \ll \frac{1}{\delta}.
\end{equation}
Under these conditions, and still assuming $|\omega - NP| <\delta$, we have $[x,b(\omega)] \gtrsim [x,b(\omega_{NP})] = e^{-s}$.

Assuming these conditions are satisfied and turning to point (2),we estimate
$$
\left| \frac{\partial_{\omega_2} h_{\lambda,\omega}(x)}{h_{\lambda,\omega}(x)} \right| = \left| (i\lambda - \rho) \frac{\partial_{\omega_2} [x,b(\omega)]}{[x,b(\omega)]} \right| \lesssim \left| (i\lambda - \rho) \frac{\partial_{\omega_2} [x,b(\omega)]}{e^{-s}} \right|
\lesssim \lambda \left[ (1+e^{2s})\delta + \delta |v|^2 + |v_1| \right].
$$
This is $\ll \frac{1}{\delta}$ if 
$$
s < -\frac{1}{2} \log (\lambda \delta^2) - C, \quad \mbox{and} \quad |v| \ll \frac{1}{\lambda\delta},
$$
which in particular implies~\eqref{conditionsv}. Under these conditions, we have $|h_{\lambda, \omega}(x)| \gtrsim |h_{\lambda, \omega_{NP}}(x)| = e^{-s}$, which implies that
$$
[E_\lambda \varphi_\delta](x) \gtrsim \lambda^\rho \delta^{d-1} e^{\rho s}.
$$
We can now compute a lower bound for the $L^p$ norm of $E_\lambda \varphi_\delta$, using the expression for the metric in Iwasawa coordinates:
$$
\| E_\lambda \varphi_\delta \|_{L^p}^p \gtrsim (\delta^{d-1} \lambda^\rho)^p \int_{-\infty}^{- \frac{1}{2} \log ( \lambda \delta^2)} e^{\rho p s} \left( \frac{e^{-s}}{\lambda \delta} \right)^{d-1} \,ds \sim
\frac{1}{p-2} \lambda^{\rho(\frac{p}{2} - 1)} \delta^{\rho p} .
$$
\end{proof}

\section{Boundedness of spectral projectors: proof of Theorem \ref{thm-projector}} \label{pf-restr}

\subsection{Duality line}

\noindent \underline{Step 1: real interpolation.}
Let us denote
\begin{align*}
P_{\Lambda}f = \delta_{\Lambda}(D)f := \vert \textbf{c}(\Lambda) \vert^{-2} f * \Phi_{\Lambda} ,
\end{align*}
where $\Phi_{\Lambda}$ is defined in \eqref{sphericalfunction}. 

We decompose $\delta_{\Lambda}(\lambda)$ into
$$
\delta_\Lambda(\lambda) = 2^{k_0} \chi(2^{k_0}(\lambda - \Lambda))- \sum_{k \geq k_0} 2^{k} \psi(2^k(\lambda-\Lambda)),
$$
where 
$$
\chi \in \mathcal{S}, \quad \int \chi = 1, \quad \psi = \chi - 2 \chi(2 \cdot), \quad 2^{k_0} \sim \frac{1}{\Lambda}.
$$
Furthermore, we choose $\chi$ such that $\widehat{\psi}$ is supported on an annulus. In order to use formulas~\eqref{albatros1} and~\eqref{albatros2}, we need even multipliers, so that we will actually write
$$
\delta_{\Lambda}(\lambda) + \delta_{-\Lambda}(\lambda) = \underbrace{2^{k_0} \chi(2^{k_0}(\lambda - \Lambda)) + 2^{k_0} \chi(2^{k_0}(-\lambda - \Lambda))}_{\displaystyle \mathcal{J}_{\Lambda,k_0}(\lambda)} - \sum_{k \geq k_0}  \underbrace{2^{k} \psi(2^k(\lambda-\Lambda)) + 2^{k} \psi(2^k(-\lambda-\Lambda))}_{\displaystyle  \mathcal{K}_{\Lambda,k_0}(\lambda)}.
$$
We will denote the convolution kernels of $\mathcal{J}_{\Lambda,k}(D)$ and $\mathcal{K}_{\Lambda,k}(D)$ by
$$
J_{\Lambda,k}(r) \qquad \mbox{and} \qquad K_{\Lambda,k}(r)
$$
respectively.

The operator norm of $\mathcal{K}_{\Lambda,k}(D)$ can be bounded as follows:
\begin{itemize}
\item By Plancherel's theorem, $\| \mathcal{K}_{\Lambda,k}(D) \|_{L^2 \to L^2} \lesssim 2^k$.
\item Turning to ${L^1 \to L^\infty}$ bounds, they can be obtained thanks to Lemma~\ref{busard}
$$
\| \mathcal{K}_{\Lambda,k}(D) \|_{L^1 \to L^\infty} \lesssim \| K_{\Lambda,k}(r) \|_{L^\infty} \lesssim \Lambda^\rho (\sinh(c 2^k))^{-\rho}.
$$
\end{itemize}
Interpolating between these two estimates,
\begin{align} \label{main-interp}
\| \mathcal{K}_{\Lambda,k}(D) \|_{L^{p'} \to L^p} \lesssim 2^{\frac{2}{p}k} \left[ {\Lambda} (\sinh(c 2^k))^{-1}\right]^{\big(1-\frac{2}{p}\big)\rho}.
\end{align}
Proceeding analogously,
$$
\| \mathcal{J}_{\Lambda,k_0}(D) \|_{L^{p'} \to L^p} \lesssim \Lambda^{-\frac{2}{p}} \Lambda^{(d-1)\big(1-\frac{2}{p}\big)}.
$$
Therefore, as long as $p \in [2,p_{ST}) \cup (p_{ST},\infty]$,
\begin{align*}
\| P_\Lambda \|_{L^{p'} \to L^p} & \lesssim \sum_{k \geq k_0} 2^{\frac{2}{p}k} \left[ \frac{\Lambda}{\sinh (c2^k)} \right]^{\big(1-\frac{2}{p}\big)\rho} \\
&  \lesssim   \Lambda^{\big(1-\frac{2}{p} \big)\rho} \bigg[ \sum_{k=k_0}^0 2^{\frac{2}{p}k - \big(1 - \frac{2}{p} \big)\rho k}  + \sum_{k=0}^{\infty} 2^{\frac{2}{p}k} \sinh(c 2^k)^{-\big(1 - \frac{2}{p} \big)\rho } \bigg]
\\ 
& \lesssim  \Lambda^{\big(1-\frac{2}{p} \big)\rho} \sum_{k=k_0}^0 2^{\frac{d-1}{2} \big( \frac{p_{ST}}{p}-1 \big) k} \\ 
&  \lesssim\left\{ \begin{array}{lll}
\Lambda^{d-1-\frac{2d}{p}} & \mbox{if $p > p_{ST} = \frac{2(d+1)}{d-1}$} \\
(1 + \ln \Lambda)\Lambda^{\big(1 - \frac{2}{p_{ST}} \big)\rho}  & \mbox{if $p = p_{ST} $}
\\
\left[1+(p-2)^{-1}\right] \Lambda^{(d-1) \left( \frac{1}{2} -\frac{1}{p} \right) } & \mbox{if $2 < p < p_{ST}$}
\end{array} \right. 
\end{align*}
The sharp dependence on $p$ is obtained using that $\sup_{x>0} x^{\alpha} e^{-\beta x} = \big( \frac{\alpha}{\beta} \big)^{\alpha} e^{-\alpha}$ where $\alpha,\beta>0.$

To remove the logarithmic singularity when $p=p_{ST},$ we use a complex interpolation argument.

\bigskip
\noindent
\underline{Step 2: complex interpolation.} Define the analytic family of operators
$$
Q_{\Lambda}(s) =\big(2^s - \frac{1}{2} \big) \sum_{k \geq k_0-1} 2^k  2^{ks} \left[ \psi(2^k(D-\Lambda)) + \psi(2^k(-D-\Lambda))\right]
$$
It is such that
\begin{itemize}
\item $Q_\Lambda(0) = \frac{1}{2} \big( P_\Lambda - \mathcal{J}_{\Lambda,k_0} \big)$. 
\item $Q_{\Lambda}$ is $i\frac{2 \pi}{\log 2} \mathbb{Z}-$periodic in $s.$
\item If $\mathfrak{Re}(s) = -1$, then
\begin{align*}
\Vert Q_{\Lambda} (s) \Vert_{L^2 \to L^2} \lesssim \Bigg \Vert (2^{i \Im s} -1) \sum_{k \geqslant k_0} 2^{i(\Im s)k} \psi(2^k \lambda) \Bigg \Vert_{L^{\infty}_{\lambda}},
\end{align*}
slightly abusing notations since we do not distinguish between $\psi$ and its even part.

Finally for any integer $N$ we have 
\begin{align*}
\bigg \vert  (2^{i \Im s} -1) \sum_{N \geqslant k \geqslant k_0} 2^{i(\Im s)k} \psi(2^k \lambda) \bigg \vert & = \bigg \vert 2^{i \Im s} \big( \psi(2^N \lambda) - \psi(2^{k_0 \lambda} \big) - \sum_{N \geqslant k \geqslant k_0-1} 2^{i \Im s} \int_{2^{k-1}}^{2^k} \psi'(\mu) d\mu \bigg \vert \\
& \lesssim 2 \Vert \psi \Vert_{L^{\infty}} + \Vert \psi' \Vert_{L^1},
\end{align*}
and we conclude that $Q_{\Lambda}(s)$ is a bounded $L^2-L^2$ operator, uniformly on $s.$
\item If $\mathfrak{Re}(s) = \rho$, $Q_\Lambda(s)$ maps $L^1$ to $L^\infty$ with operator norm $O(\Lambda^\rho)$ (uniformly in $\mathfrak{Im} s$) since, by Lemma~\ref{busard} below,
$$
\left\| \sum_{k \geq k_0-1} 2^{ks} K_{\Lambda,k}(r) \right\|_{L^\infty} \lesssim \left\| \sum_{k \geq k_0-1} 2^{\rho k} |K_{\Lambda,k}(r)| \right\|_{L^\infty} \lesssim \Lambda^\rho.
$$
\end{itemize}
The desired bound for $p = p_{ST}$ follows by complex interpolation, see for example Theorem 4.1, Chapter 5 in \cite{SW}.

\begin{lemma}[Pointwise kernel bounds] \label{busard}
For $2^k \gtrsim \Lambda^{-1} $, the following pointwise bounds hold:
\begin{align*}
&  \big \vert J_{\Lambda,k_0} (r) \big \vert \lesssim \Lambda^{d-1} ,\\
& \big \vert K_{\Lambda,k} (r) \big \vert \lesssim \bigg(\frac{\Lambda}{\sinh(c2^k)} \bigg)^{\rho} \left[ \textbf{1}_{2^k c \leqslant r \leqslant 2^k C} + \textbf{1}_{ r \leqslant 2^k c} (2^{-k} \Lambda^{-1})^{10} \right],
\end{align*}
for any $r>0$ and some positive constants $c,C$. 
\end{lemma}

\begin{proof} 
There are two cases to consider, depending on the parity of the space dimension. We start with the easier case where it is odd.

\medskip
\underline{Case 1:  $d$ is odd.} We learn from~\eqref{albatros1} that
$$
K_{\Lambda,k}(r) = \frac{1}{\sqrt{2\pi}} \left( \frac{-1}{2\pi} \frac{1}{\sinh r} \frac{\partial}{\partial r} \right)^\rho \left(\widehat{\psi}(2^{-k} r) e^{ir\Lambda} + \widehat{\psi}(-2^{-k} r) e^{-ir\Lambda}\right).
$$
Given our assumption on the support of $\widehat{\psi},$ we have $r \sim 2^k.$ 

Using \eqref{decompmder} and \eqref{boundFl} we obtain 
\begin{align*}
\bigg \vert K_{\Lambda,k}(r) \bigg \vert \lesssim \bigg( \frac{\Lambda}{\sinh(c 2^k)} \bigg)^{\rho} \sum_{l=0}^{\rho} \big(2^k \Lambda\big)^{l-\rho} \big( 2^{-k} e^{-r} \sinh r \big)^{l-\rho} \lesssim \bigg( \frac{\Lambda}{\sinh(c 2^k)} \bigg)^{\rho},
\end{align*}
where we used the assumption $2^k \gtrsim \Lambda^{-1}$ to obtain the last inequality. 

\medskip
\underline{Case 2: $d$ is even.} Using \eqref{albatros2} together with \eqref{decompmder} we can write that up to unimportant numerical constants,
\begin{align} \label{defI}
\begin{split}
\displaystyle K_{\Lambda,k}(r) & =  \sum_{l=0}^{d/2} I_{l,d/2,\Lambda,k} (r) \\
 I_{l,d/2,\Lambda,k}(r) & :=\frac{1}{\sqrt{ \pi}} \int_r^\infty F_{l,d/2}(s) \partial_s^l \left(\widehat{\psi}(2^{-k} s) e^{is\Lambda} + \widehat{\psi}(-2^{-k} s) e^{-is\Lambda} \right) (\cosh s - \cosh r)^{-1/2} 
 \sinh s \,ds.
\end{split} 
 \end{align}
Note that due to the compact support assumption on $\widehat{\psi},$ the integrals $ I_{l,d/2,\Lambda,k}$ are equal to $0$ when $r \geqslant 2^{k+4}.$ 

Therefore we can safely assume that $r<2^{k+4}.$ The worst region corresponds to $2^{k-4}< r<2^{k+4}.$ We focus on this case now.

We insert additional localizers and split the above integral in three parts:
\begin{align*}
 I_{l,d/2,\Lambda,k}(r) & := \sum_{j=1}^3  I_{l,d/2,\Lambda,k}^{(j)}(r) \\
  I_{l,d/2,\Lambda,k}^{(j)}(r) & :=  \frac{1}{\sqrt{ \pi}} \int_r^\infty F_{l,d/2}(s) \partial_s^l \left(\widehat{\psi}(2^{-k} s) e^{is\Lambda} + \widehat{\psi}(-2^{-k} s) e^{-is\Lambda} \right) (\cosh s - \cosh r)^{-1/2} 
 \chi_j(s-r) 
 \sinh s \,ds, 
 \\
\chi_1 (s) & := \phi(\Lambda s), \ \ \chi_2(s) :=  \big(1 - \phi(\Lambda s) \big) \phi(s), \  \ \chi_3(s) := 1-  \phi(s). 
\end{align*}
For the first piece $ I_{l,d/2,\Lambda,k}^{(1)}$, we use that $\cosh s - \cosh r = 2 \, \sinh \frac{s+r}{2} \, \sinh \frac{s-r}{2} \gtrsim (s-r) \, \sinh \frac{s}{2}.$ Indeed due to the localizer $\chi^{(1)},$ we have $0 \leqslant s-r \leqslant 2/\Lambda.$ 
We integrate directly using \eqref{boundFl} and obtain
\begin{align*}
\big \vert  I_{l,d/2,\Lambda,k}^{(1)}(r) \big \vert & \lesssim  \Lambda^l \int_r ^{\infty} \vert F_{l,d/2}(s) \vert \vert \widehat{\psi} (2^{-k} s ) \vert \chi_1(s-r) \, \sinh s  \frac{ds}{\sqrt{s-r} \, \sqrt{\sinh (s/2)}} 
\\
& \lesssim \Lambda^l \int_r ^{\infty} \frac{\vert \widehat{\psi}(2^{-k}s) \vert}{\sqrt{\sinh (s/2)}} (\sinh s)^{-d/2+1} \cdot (\sinh s)^{-d/2-1+l}  e^{s(d/2-l)} \, \chi_1(s-r)  \, \sinh s \, \frac{ds}{\sqrt{s-r}}
\\
& \lesssim  \bigg(\frac{\Lambda}{\sinh(c 2^k)} \bigg)^{\rho} \Lambda^{l-\rho}  \int_r ^{\infty} s^{d/2-l} \vert \widehat{\psi}(2^{-k} s ) \vert \cdot   \big( s^{l-d/2} (\sinh s)^{-d/2+l}  e^{s(d/2-l)} \big) \, \chi_1(s-r)  \, \frac{ds}{\sqrt{s-r}}
\\
& \lesssim  \bigg(\frac{\Lambda}{\sinh(c 2^k)} \bigg)^{\rho} \, \Lambda^{l-\rho} \, 2^{k (d/2-l)} \underbrace{\int_r ^{\infty}  \big( s^{l-d/2} (\sinh s)^{-d/2+l}  e^{s(d/2-l)} \big) \, \chi_1(s-r)  \, \frac{ds}{\sqrt{s-r}}}_{\lesssim \Lambda^{-1/2}}
\\
& \lesssim \bigg(\frac{\Lambda}{\sinh(c 2^k)} \bigg)^{\rho} \big(2^{-k} \Lambda^{-1} \big)^{\frac{d}{2}-l}.
\end{align*}
The second factor is summable in $l$ given the condition $2^k \gtrsim \Lambda^{-1}$, and the desired bound follows.

For the second piece, we integrate by parts once in $s.$ Using that
\begin{align*}
&\Bigg \vert  \partial_s \bigg( \frac{\widehat{\psi}(2^{-k}s) F_{l,d/2}(s) \,  \sinh s \,  \phi(s-r)}{\sqrt{\cosh s - \cosh r}} \bigg) \Bigg \vert \\
\lesssim &\bigg[2^{-k} + \frac{e^s}{\sinh s} + \frac{\sinh s}{\cosh s - \cosh r} \bigg] \vert \widehat{\psi}(2^{-k} s) \vert \big(\sinh s \big)^{-d+l+1} e^{s\big(\frac{d}{2} - l \big)} \vert \phi(s-r) \vert ,
\end{align*}
we deduce the bound
\begin{align*}
\big \vert  I_{l,d/2,\Lambda,k}^{(2)}(r) \big \vert & \lesssim \Lambda^l \Bigg \vert \int_r^{\infty} \frac{e^{i\Lambda s}}{i\Lambda} \partial_s \bigg( \frac{\widehat{\psi}(2^{-k}s) F_{l,d/2}(s) \,  \sinh s \, \phi(s-r)}{\sqrt{\cosh s - \cosh r}} \bigg) \big(1 - \phi(\Lambda(s-r)) \big) \, ds \Bigg \vert \\
&\lesssim \bigg(\frac{\Lambda}{\sinh(c 2^k)} \bigg)^{\rho} \big(2^{-k} \Lambda^{-1} \big)^{\frac{d}{2}-l} \big[ 2^{-k} \Lambda^{-1} + 2^{-k} \Lambda^{-1} +  1 \big].
\end{align*}
For the last bound we used the fact that since $0 \leqslant s-r \leqslant 2, $ we have $(\cosh s - \cosh r)^{-1} \sinh s \sim (s-r)^{-1}. $ Note that we discarded the case where the derivative hits the $\Lambda-$dependent localizer since that term can be treated by direct integration as for $I_{l,d/2,\Lambda,k}^{(1)}.$ 

We conclude as above.

The last piece $I_{l,d/2,\Lambda,k}^{(3)}$ is handled similarly, in fact easier since we have 
\begin{align*}
\bigg \vert \big(1- \phi(s-r) \big) \partial_s^{\alpha} \big( \frac{1}{\sqrt{\cosh s -\cosh r}} \big) \bigg \vert \lesssim e^{-s/2}.
\end{align*}

Finally when $r \leqslant 2^{k-4}.$ Coming back to \eqref{defI} we perform ten integrations by parts relying on the identity $\frac{1}{i \Lambda} \partial_s e^{is \Lambda} = e^{is \Lambda}. $ In that regime $s-r \sim s$ (indeed recall that $s \sim 2^k$), therefore we have
\begin{align*}
\bigg \vert \partial_s^{10} \bigg( \frac{ F_{l,d/2}(s) \widehat{\psi}(2^{-k} s)}{\sqrt{\cosh s - \cosh r}} \bigg) \bigg \vert \lesssim \bigg( 2^{-10k} + (\sinh s)^{-10} e^{s(d/2-l+10)}  \bigg) (\sinh s)^{-d+l}  e^{s(d/2-l)}  \vert \widehat{\psi}(2^{-k}s) \vert.
\end{align*}
This yields the bound
\begin{align*}
\vert I_{l,d/2,\Lambda,k}(r) \vert \lesssim \bigg( \frac{\Lambda}{\sinh(c 2^k)} \bigg)^{\rho} \big(2^{-k} \Lambda^{-1} \big)^{\frac{d}{2}-l+10},
\end{align*}
which in turn gives the desired result. 

Turning to $J_{\Lambda, k_0}$, we only give a brief justification in the case $d$ odd. Since $2^{k_0} \sim \Lambda^{-1}$, the formula for $J_{\Lambda,k_0}$ can be written
$$
J_{\Lambda,k_0}(r) = \left( \frac{1}{\sinh r} \frac{\partial}{\partial r} \right)^\rho \left(F(r\Lambda) + F(-r\Lambda)\right),
$$
for a function $F$ with bounded derivatives. From this expression, it is not hard to derive the desired bound.
\end{proof}

\subsection{Optimality on the duality line} Since $P_\lambda = (R_\lambda)^* R_\lambda$, it follows that
$$
\| (R_\lambda)^* \|_{L^{2} \to L^{p'}} = \| R_\lambda \|_{L^p \to L^2} = \| P_\lambda \|_{L^{p'} \to L^p}^{1/2}.
$$

By Lemma~\ref{lemmaradial}, $\Phi_\lambda \in L^p$ for $p>2$ only, so that $P_\lambda$ can only be bounded from $L^s$ to $L^q$ if $q >2$. Since $P_\lambda$ is self-adjoint, this gives furthermore the condition $s<2$.

We also learn from Lemma~\ref{lemmaradial} that, if $p > \frac{2d}{d-1}$,
$$
\| E_\lambda \|_{L^2 \to L^p} \geq \frac{|c(\lambda)|^{-1} \| \Phi_\lambda \|_{L^p(\mathbb{H}^d)} }{\| 1 \|_{L^2(\mathbb{S}^{d-1})}} \gtrsim \lambda^{\rho - \frac{d}{p}},
$$
which implies that, if $p > \frac{2d}{d-1}$,
$$
\| P_\lambda \|_{L^{p'} \to L^p} \gtrsim \lambda^{d-1-\frac{2d}{p}}.
$$

Finally, by Lemma~\ref{lemmaknapp}, for any $\delta < \lambda^{-1/2}$,
$$
\| E_\lambda \|_{L^2 \to L^p}  \geq \frac{ \| E_\lambda \varphi_\delta \|_{L^p(\mathbb{H}^d)} }{\| \varphi_\delta \|_{L^2(\mathbb{S}^{d-1})}} \gtrsim (p-2)^{-1/2}  \lambda^{ \frac{\rho}{2}-\frac{\rho}p},
$$
which implies
$$
\| P_\lambda \|_{L^{p'} \to L^p} \gtrsim (p-2)^{-1}  \lambda^{ \frac{d-1}{2}-\frac{d-1}p}.
$$

\subsection{Off-duality line}
Since $P_\lambda = E_\lambda R_\lambda$, $R_{\lambda}^{\star} = E_{\lambda} ,$ we find $\| P_\lambda \|_{L^{p'} \to L^p} = \| E_\lambda \|_{L^{2} \to L^p}^{2} =  \| R_\lambda \|_{L^{p'} \to L^2}^{2}$ and 
\begin{align*}
\| P_\lambda \|_{L^{s} \to L^{q}} & \lesssim \| R_\lambda \|_{L^{s} \to L^2} \| E_\lambda \|_{L^2 \to L^{q}} \\
& \lesssim
\left[ \lambda^{\frac{d-1}{2}-\frac{d}{q }} + \left[ ( q-2)^{-1/2} + 1 \right] 
\lambda^{\frac{(d-1)}{2} \left( \frac{1}{2} -\frac{1}{q} \right) }\right]  \left[ \lambda^{\frac{d}{s }-\frac{d+1}{2}} + \left[ (2-s )^{-1/2} + 1 \right] \lambda^{\frac{(d-1)}{2} \left( \frac{1}{ s} -\frac{1}{2} \right) }\right]  .
\end{align*}

\section{Boundedness of the resolvent: proof of Theorem \ref{thm-resolvent}}
\noindent \textit{Notation:} In this section, we denote 
\begin{align*}
\alpha(p,d) = \left\{ \begin{array}{ll}
d-1-\frac{2d}{p} & \mbox{if $p \geq p_{ST} = \frac{2(d+1)}{d-1}$} \\
(d-1) \left( \frac{1}{2} -\frac{1}{p} \right) & \mbox{if $2 < p \leq p_{ST}$}
\end{array} \right. .
\end{align*}
\subsection{Duality line}
Let $\tau>1,\varepsilon>0.$ We write
\begin{align*}
(D^2-\tau-i\varepsilon)^{-1} = \mathcal{R}_{\tau}(D) + i \mathcal{I}_{\tau}(D), \ \ 
\mathcal{R}_{\tau}(D) :=\frac{D^2-\tau}{(D^2-\tau)^2 + \varepsilon^2}, \ \ 
\mathcal{I}_{\tau}(D) := \frac{\varepsilon}{(D^2-\tau)^2 + \varepsilon^2}.
\end{align*}

\subsubsection{Treating $\mathcal{R}_{\tau}(D)$} 
Let $k_0,k_1 \in \mathbb{Z}$ be such that $2^{k_0} \sim \frac{1}{\sqrt{\tau}}, 2^{k_1} \sim \frac{1}{\delta},$ $\delta>0$ to be chosen later. Let $\chi:\mathbb{R} \rightarrow \mathbb{R}$ be such that 
\begin{align*}
\chi(0) \neq 0,  \chi \in \mathcal{S},
\end{align*} 
and we require that the support of the Fourier transform of $x \mapsto \frac{\chi(x)-\chi(2 x) }{x}$ be contained in an annulus.  We write
\begin{align*}
\frac{D^2-\tau}{(D^2-\tau)^2 + \varepsilon^2} &= \chi(0)^{-1} \bigg[ \underbrace{\big(\chi(0)- \chi(2^{k_0}(D-\sqrt{\tau})\big) \frac{D^2-\tau}{(D^2-\tau)^2 + \varepsilon^2}}_{\displaystyle \mathcal{A}_{\tau,k_0}(D)} \\ &+ \underbrace{\sum_{k=k_0}^{k_1} \big[ \chi(2^k(D-\sqrt{\tau})) - \chi(2^{k+1}(D-\sqrt{\tau})) \big]\frac{D^2-\tau}{(D^2-\tau)^2 + \varepsilon^2}}_{\displaystyle  \mathcal{B}_{\tau,k_0,k_1}(D)}\\
& + \underbrace{\chi(2^{k_1+1}(D-\sqrt{\tau}))\frac{D^2-\tau}{(D^2-\tau)^2 + \varepsilon^2}}_{\displaystyle \mathcal{C}_{\tau,k_1}(D)}  \bigg].
\end{align*}
As in the proof of Theorem \ref{thm-projector}, we are really considering the even extension of those multipliers. In the sequel we will abuse notations by still denoting $\chi$ its even part. \\
\textit{Treating $\mathcal{A}_{\tau,k_0}(D)$:} We introduce a second even cut-off $\phi:\mathbb{R} \rightarrow [0,1]$ that is equal to $1$ on $[-1,1],$ whose support is included on $[-2,2].$ We split the kernel into two pieces adding localizations $\phi(\frac{D}{10\sqrt{\tau}})$ and $1-\phi(\frac{D}{10\sqrt{\tau}})$ respectively. $A_{\tau,k_0}^{(1)}$ and $A_{\tau,k_0}^{(2)}$ denote the corresponding kernels.
We have 
\begin{align*}
\widetilde{ A_{\tau,k_0}^{(1)}} (\lambda) = \frac{\lambda^2-\tau}{(\lambda^2-\tau)^2 + \varepsilon^2} \phi\big(\frac{\lambda}{10\sqrt{\tau}} \big) \big(\chi(0) -\chi(2^{k_0}(\lambda-\sqrt{\tau})) \big).
\end{align*}
Using the assumption on the support of $\phi$ we write 
\begin{align*}
\big \vert \widetilde{ A_{\tau,k_0}^{(1)}} (\lambda) \big \vert & \lesssim 2^{k_0} \cdot \Bigg \vert (\lambda - \sqrt{\tau}) \frac{\lambda^2 - \tau}{(\lambda^2 - \tau)^2+\varepsilon^2}  \phi\big(\frac{\lambda}{10\sqrt{\tau}} \big) \Bigg \vert \cdot \Bigg \vert \frac{\chi(0)-\chi(2^{k_0}(\lambda - \sqrt{\tau}))}{2^{k_0}(\lambda - \sqrt{\tau})} \Bigg \vert \\
& \lesssim 2^{k_0} \cdot \sqrt{\tau} \cdot \Vert \chi' \Vert_{L^{\infty}} .
\end{align*}
Therefore we can conclude that 
\begin{align} \label{A-1-fourier}
\Vert \widetilde{ A_{\tau,k_0}^{(1)}} \Vert_{L^{\infty}_{\lambda}} \lesssim \tau^{-1}.
\end{align}
Note that to derive this bound we make crucial use of the $\chi$ factor. \\
Next we obtain a pointwise bound on the kernel $A_{\tau,k_0}^{(1)} .$ 

Let $m$ denote the multiplier that corresponds to the kernel $A_{\tau,k_0}^{(1)}.$ It is defined through its flat Fourier transform as
\begin{align*}
\widehat{m}(r) = \int_{\mathbb{R}} \cos (\lambda r) \frac{\lambda^2 - \tau}{(\lambda^2 - \tau)^2 + \varepsilon^2} \phi \big(\frac{\lambda}{10 \sqrt{\tau}} \big) \big( \chi(0) - \chi(2^{k_0}(\lambda - \sqrt{\tau})) \big) d\lambda.
\end{align*}

Therefore using \eqref{decompmder} and \eqref{boundFl} we find
\begin{align} \label{boundrlarge}
\notag \big \vert F_{l,M}(r) \partial_r ^l \widehat{m} \big \vert & \lesssim (\sinh r)^{-2 M + l} e^{r(M-l)}  \int_{\mathbb{R}} \vert  \lambda \vert ^l \phi \big(\frac{\lambda}{10 \sqrt{\tau}} \big) \bigg \vert \frac{\lambda^2 - r}{(\lambda^2 -r)^2 + \varepsilon^2} \big(\chi(0) - \chi(2^{k_0}(\lambda - \sqrt{\tau})) \big) \bigg \vert   d\lambda  \\
& \lesssim (\sinh r)^{-2M + l} e^{r(M-l)} \tau^{\frac{l-1}{2}} .
\end{align}
This estimate is sufficient in the regime $r \geqslant 1/\sqrt{\tau}.$ To deal with the regime $r<1/\sqrt{\tau}$ we derive a second pointwise estimate.

For the odd indices $l$ in the corresponding sum \eqref{decompmder}, we have that $\partial_r ^l \widehat{m} $ equals
\begin{align*}
& \partial_r ^l \bigg( \int_{\mathbb{R}} \bigg( \cos (\lambda r) - (-1)^{\frac{l-1}{2}} \frac{(r\lambda)^{l}}{l!}  \sum_{u=0}^{N} (-1)^u \frac{r^{2u} \lambda^{2u}}{(2u)!} \bigg)  \frac{\lambda^2-\tau}{(\lambda^2-\tau)^2+\varepsilon^2} \phi(\frac{\lambda}{10 \sqrt{\tau}}) \big(\chi(0) - \chi(2^{k_0}(\lambda-\sqrt{\tau}))\big) d\lambda \bigg) \\
&=r^{2\rho-l} \int_{\mathbb{R}} (-1)^{\frac{l-1}{2}} \lambda^{l} \frac{ \cos (\lambda r) -\sum_{u=0}^{N} (-1)^u \frac{(r \lambda)^{2u}}{(2u)!}}{r^{2\rho-l}}  \frac{\lambda^2-\tau}{(\lambda^2-\tau)^2+\varepsilon^2} \phi(\frac{\lambda}{10 \sqrt{\tau}}) \big(\chi(0) - \chi(2^{k_0}(\lambda-\sqrt{\tau}))\big) d\lambda.
\end{align*}
From this expression we deduce the pointwise bound (relying on \eqref{boundFl})
\begin{align} \label{boundtermm}
\big \vert F_{l,M}(r) \partial_r ^l \widehat{m} \big \vert \lesssim \tau^{\frac{l}{2}-\frac{1}{2}} (r \tau^{\frac{1}{2}})^{N'} (\sinh r)^{-2M+l} e^{r(M-l)} 
\end{align}
for any integer $N'.$ 

Taking $N' = 2M-l$ we find
\begin{align*}
\big \vert F_{l,M}(r) \partial_r ^l \widehat{m} \big \vert \lesssim \tau^{M-1/2} \big(\frac{r}{\sinh r} \big)^{2M-l} e^{r(M-l)}.
\end{align*}
For the even indices $l,$ we write 
\begin{align*}
\partial_r ^l \widehat{m} = \sinh r  \partial_r \big( \frac{1}{\sinh r} \partial_r^{l-1} \widehat{m} \big) - \sinh r \partial_r \big( \frac{1}{\sinh r} \big) \partial_r ^{l-1} \widehat{m}. 
\end{align*}
Using the expression above for $l-1,$ we deduce the same pointwise bound on $F_{l,M}(r) \partial_r ^l \widehat{m}$ in this case. 

In the case where $d$ is odd this, together with \eqref{albatros1} and \eqref{decompmder} with $M=\rho$, yields the bound 
\begin{align*}
\Vert A_{\tau,k_0}^{(1)} \Vert_{L^{\infty}_r} \lesssim \tau^{\frac{d}{2}-1}.
\end{align*}
We move on to the case of even dimension. 

In the regime $r<1/\sqrt{\tau}$ we write, using \eqref{boundtermm} for $M=d/2$ and \eqref{boundrlarge}
\begin{align*}
\vert A_{\tau,k_0}^{(1)}(r) \vert \lesssim \tau^{\frac{d-1}{2}}  \int_r^{r+1} 
\bigg(\frac{s}{\sinh s} \bigg)^{d-l} e^{s(\frac{d}{2}-l)} 
\min(1,s\sqrt{\tau})  \frac{\sinh s}{\sqrt{\cosh s -\cosh r}} ds + \tau^{\frac{d-2}{4}} \int_{r+1}^{\infty} e^{-d/2} \frac{\sinh s}{\sqrt{\cosh s - \cosh r}} ds.
\end{align*}
Next we further decompose the integral into
\begin{align*}
\int_r^{\infty} \min(1,s\sqrt{\tau})  \frac{\sinh s}{\sqrt{\cosh s -\cosh r}} ds 
\lesssim \int_r ^{\frac{10}{\sqrt{\tau}}} s \sqrt{\tau} \frac{\sinh s}{\sqrt{(s-r)\sinh s} } ds +  \int_{\frac{10}{\sqrt{\tau}}}^{1}  \frac{\sinh s}{\sqrt{(s-r) \sinh s} } ds \lesssim \tau^{-1/2}.
\end{align*} 
In the regime $r\geqslant 1/\sqrt{\tau}$ we can estimate more crudely, using \eqref{boundrlarge}
\begin{align*}
\vert A_{\tau,k_0}^{(1)}(r) \vert \lesssim  \tau^{\frac{d-2}{4}} \int_r^{r+1} (\sinh s)^{-\frac{d}{2}} \frac{\sinh s}{\sqrt{(s-r)\sinh r}} ds +  \tau^{\frac{d-2}{4}} \int_{r+1}^{\infty} e^{\frac{1-d}{2}s} ds \lesssim \tau^{\frac{d}{2}-1}.
\end{align*}
Overall we obtain the same bound as in the $d$ even case.
For the second piece $\mathcal{A}_{\tau,k_0}^{(2)}$ 
In this case the flat Fourier transform of the multiplier is given by
\begin{align}\label{mhat-far}
\widehat{m}(r) = \int_{\mathbb{R}} \cos(\lambda r) \frac{\lambda^2 - \tau}{(\lambda^2 - \tau)^2 + \varepsilon^2} \big(1 - \phi \big(\frac{\lambda}{10 \sqrt{\tau}} \big) \big) \big(\chi(0) - \chi(2^{k_0}(\lambda - \sqrt{\tau})) \big) d\lambda.
\end{align}
We have two cases to consider depending on the dimension, the case of the hyperbolic plane being slightly different from the point of view of Hardy-Littlewood-Sobolev inequalities.

\underline{Case 1: $d \geqslant 3$} We use the pointwise estimate
\begin{align} \label{point-far-3}
\big \vert A_{\tau,k_0}^{(2)}(r) \big \vert \lesssim \frac{\tau^{\frac{\rho-1}{2}}}{(\sinh r)^{\rho} (r \sqrt{\tau})^{\beta} } ,  \ \ \beta \in [\rho-1,\rho+2],
\end{align} 
deduced from \eqref{albatros1} and \eqref{albatros2} by performing integrations by parts in \eqref{mhat-far}.  Note that we make crucial use of the localizer $1-\phi$ at this stage, since it is necessary to be away from the principal value singularity. \\ 
Fix $2 <p \leqslant \frac{2d}{d-2}.$ Note that the upper bound guarantees that $1/p'-1/p \leqslant 2/d.$ The pointwise bound \eqref{point-far-3} implies the desired $L^{p'}-L^p$ boundedness. 
Indeed for the piece $r \leqslant 1$ we use a local version of the Hardy-Littlewood-Sobolev inequality (which relies on \eqref{radial-conv}) and the pointwise estimate \eqref{point-far-3} with $\beta = \frac{1}{2} + d \big(\frac{2}{p} - \frac{1}{2} \big).$ Indeed in that region and for this choice of exponent, \eqref{point-far-3} is upper bounded by $ \frac{\tau^{\frac{d}{2} \big(1-\frac{2}{p} \big)-1}}{r^{\frac{2d}{p}}}.$ 
The piece $r >1$ is more favorable: we use Lemma \ref{K-S} with $q = \widetilde{q} >2$. Overall we obtain that
$\Vert \mathcal{A}_{\tau,k_0}^{(2)}(D) \Vert_{L^{p'} \rightarrow L^p} \lesssim \tau^{\frac{d}{2} \big(1-\frac{2}{p} \big)-1}.$  

\underline{Case 2: $d=2$} We see from \eqref{albatros2} that
\begin{align} \label{point-far-2}
\textrm{If} \ 0 < r \leqslant 1, \ \big \vert A_{\tau,k_0}^{(2)}(r) \big \vert & \lesssim \vert \ln r \vert, \frac{1}{r^{\beta} \sqrt{\tau}^{\beta}}, \textrm{for} \ \beta \in (0,1], \\
\textrm{If} \ r>1, \ \big \vert A_{\tau,k_0}^{(2)}(r) \big \vert & \lesssim \frac{1}{ (r\tau)^2 \sqrt{\sinh r}},
\end{align}
hence using as above the Hardy-Littlewood-Sobolev inequality in the first region and Lemma \ref{K-S} in the second, we deduce that for $2 < p < \infty$
\begin{align*}
\Vert \mathcal{A}_{\tau,k_0}^{(2)} \Vert_{L^{p'} \rightarrow L^p} \lesssim \tau^{-\frac{2}{p}}.
\end{align*}

\bigskip
\noindent
\textit{Treating $\mathcal{B}_{\tau,k_0,k_1}(D)$:} The idea here is that, for $\frac{\epsilon}{\sqrt{\tau}} \ll |D-\tau| \ll \sqrt{\tau}$, one can approximate 
$$
\frac{D^2 - \tau}{(D^2 - \tau)^2 +\epsilon^2} \sim \frac{1}{2 \sqrt{\tau} (D-\sqrt{\tau})}.
$$
To formalize this, we define the approximation error after localization as 
\begin{align*}
R & := \big[ \chi(2^k(D-\sqrt{\tau})) - \chi(2^{k+1}(D-\sqrt{\tau})) \big]\frac{D^2-\tau}{(D^2-\tau)^2 + \varepsilon^2} - \frac{1}{2 \sqrt{\tau}} \big[2^k \psi(2^k(D-\sqrt{\tau})) - 2^{k+1} \psi(2^{k+1}(D-\sqrt{\tau})) \big], 
\end{align*}
where $\psi(x) = \frac{\chi(x)-\chi(0)}{x}. $ 

Therefore we have a decomposition of the localized resolvent operator into a main component and a remainder:
\begin{align*}
\big[ \chi(2^k(D-\sqrt{\tau})) - \chi(2^{k+1}(D-\sqrt{\tau})) \big]\frac{D^2-\tau}{(D^2-\tau)^2 + \varepsilon^2} =  \frac{2^{k-1}}{\sqrt{\tau}} \Psi(2^k(D-\sqrt{\tau})) + R ,
\end{align*}
where $\Psi(x) = \psi(x) - 2 \psi(2x).$ 

We further decompose the error as
\begin{align*}
R& = R_1 - R_2 \\
R_1 & :=- \big[ \chi(2^k(D-\sqrt{\tau})) - \chi(2^{k+1}(D-\sqrt{\tau})) \big] \frac{(D^2-\tau)(D-\sqrt{\tau})}{2 \sqrt{\tau} \big[ (D^2 - \tau)^2 + \varepsilon^2 \big] } \\  
R_2 & : = 2^k \Psi(2^k(D-\sqrt{\tau}))  \frac{\varepsilon^2}{2 \sqrt{\tau} \big[ (D^2 - \tau)^2 + \varepsilon^2 \big]}.
\end{align*} 
To handle the main term, we use \eqref{main-interp}:
\begin{align*}
\sum_{k = k_0}^{k_1} \bigg \Vert  \frac{2^{k-1}}{\sqrt{\tau}} \Psi(2^k(D-\sqrt{\tau})) \bigg \Vert_{L^{p'} \rightarrow L^p} \lesssim \tau^{\frac{\alpha(p,d)-1}{2}}.
\end{align*}
We can treat the remainder term $R$ using functional calculus and the bounds from Theorem \ref{thm-projector}. More precisely,
\begin{align*}
\Vert R_1 \Vert_{L^{p'} \rightarrow L^p} & \lesssim \int_0 ^{\infty} \left( |\chi(2^{k_0}(\lambda - \sqrt{\tau})) | + |\chi(2^{k_1}(\lambda - \sqrt{\tau}))| \right) \frac{1}{\sqrt{\tau}(\lambda + \sqrt{\tau})} \Vert P_{\lambda} \Vert_{L^{p'}\rightarrow\ L^p} d\lambda \\
& \lesssim \tau^{\frac{\alpha(p,d)-1}{2}} .
\end{align*}
Similarly for the second piece,
\begin{align*}
\Vert R_2 \Vert_{L^{p'} \rightarrow L^p} & \lesssim\bigg \Vert  \big(2^{k_0} \psi(2^{k_0}(D-\sqrt{\tau}))-2^{k_1} \psi(2^{k_1}(D-\sqrt{\tau})) \big) \frac{\varepsilon^2}{2 \sqrt{\tau} \big[ (D^2 - \tau)^2 + \varepsilon^2 \big]} \bigg \Vert_{L^{p'} \rightarrow L^p} \\
&\lesssim  \frac{\varepsilon \tau^{\frac{\alpha(p,d)-1}{2}}}{\sqrt{\tau} \delta} .
\end{align*}
\textit{Treating $\mathcal{C}_{\tau,k_1}(D)$:} Again, this lower order term can be bounded using functional calculus for $D$ and Theorem \ref{thm-projector}:
\begin{align*}
\Vert \mathcal{C}_{\tau,k_1}(D)\Vert_{L^{p'} \rightarrow L^p} & \lesssim \int_0 ^{\infty} \bigg \vert \frac{\chi(2^{k_1}(\lambda-\sqrt{\tau}))}{\lambda^2-\tau+i\varepsilon} \bigg \vert \Vert  P_{\lambda} \Vert_{L^{p'} \rightarrow L^p} d\lambda  \lesssim \frac{\delta \tau^{\alpha(p,d)/2}}{\varepsilon}.
\end{align*}
Putting all the estimates together and choosing $\delta = \varepsilon/\sqrt{\tau},$ the desired result follows.
\subsubsection{Treating $\mathcal{I}_{\tau}(D)$} \label{imagin}
Let $\phi:\mathbb{R}^{+} \rightarrow [0,1],$ whose support is included in $[1/2,2]$ and equal to $1$ on $[2/3,3/2].$ Abusing notations, we denote $\phi$ its even extension to $\mathbb{R}.$ Next, we decompose the operator $\mathcal{I}_{\tau}(D)$ into $\mathcal{I}_{\tau}^{(1)}(D)$ and $\mathcal{I}_{\tau}^{(2)}(D)$ by introducing localizations $\phi(D/10 \sqrt{\tau})$ and $1-\phi(D/10\sqrt{\tau}).$ \\
To deal with the singular part, we write that, using functional calculus for $D,$ changing variables and using Theorem \ref{thm-projector},
\begin{align*}
\Vert \mathcal{I}_{\tau}^{(1)}(D) \Vert_{L^{p'} \rightarrow L^p} & \lesssim \int_{0}^{\infty} \frac{\varepsilon \phi \big(\frac{\lambda}{10\sqrt{\tau}}\big)}{(\lambda^2-\tau)^2 + \varepsilon^2} \Vert P_{\lambda} \Vert_{L^{p'} \rightarrow L^p} d\lambda \\
& \lesssim \int_{-\sqrt{\tau}} ^{\infty} \frac{\phi \big( \frac{\varepsilon \lambda + \sqrt{\tau}}{10\sqrt{\tau}} \big)}{1+\lambda^2(\varepsilon \lambda + \sqrt{\tau})^2} (\varepsilon \lambda + \sqrt{\tau})^{\alpha(p,d)} d\lambda \\
& \lesssim \tau^{\frac{\alpha(p,d)-1}{2}} .
\end{align*}
The last piece is treated similarly, we obtain
$\Vert \mathcal{I}_{\tau}^{(2)}(D) \Vert_{L^{p'} \rightarrow L^p} \lesssim \tau^{\frac{d}{2} \big(1-\frac{2}{p} \big)-1} $ for $1/p'-1/p \leqslant 2/d.$
\subsection{Off-duality line}
The proof is similar to the duality line case. Therefore we only detail the parts of the argument that are different. We use the same decomposition, and keep the notations of the previous subsection.
To bound the singular part, we write (see \cite{G} for a similar reasoning in the euclidean case) using Plancherel's theorem and Theorem \ref{thm-projector} that
\begin{align*}
\Vert 2^k \Psi(2^k(\lambda-\sqrt{\tau})) \widetilde{f}(\lambda,\omega) \Vert_{L^2}^2 & =2^{2k} \int_{0}^{\infty} \vert \Psi(2^k(\lambda-\sqrt{\tau})) \vert^2 \Vert R_{\lambda} f \Vert_{L^2_{\omega}}^2 d\lambda \\
& \lesssim 2^k \Vert f \Vert_{L^{p'}}^2 \left\{ \begin{array}{ll}
\tau^{-\frac{d+1}{2}+\frac{d}{p'}} & \mbox{if $p \geq p_{ST} = \frac{2(d+1)}{d-1}$} \\
\tau^{\frac{d-1}{2}\big(\frac{1}{p'}-\frac{1}{2} \big)} & \mbox{if $2 < p \leq p_{ST}$}
\end{array} \right. .
\end{align*}
Next we interpolate this inequality with the $(L^1-L^{\infty})$ estimate that is a direct consequence of Lemma \ref{busard}, and obtain that for $1 \leqslant s \leqslant \frac{q}{q-1}, 2 \leqslant q,$
\begin{align*}
\Vert 2^k \Psi(2^k(D-\sqrt{\tau})) \Vert_{L^{s} \rightarrow L^q} \lesssim \frac{2^{\frac{k}{q}}}{\big(\sinh (c 2^k) \big)^{\rho\big(1-\frac{2}{q} \big)}}
\left\{ \begin{array}{ll}
\tau^{\frac{d}{2}\big(\frac{1}{s} - \frac{1}{2} \big)-\frac{1}{4}} & \mbox{if $\frac{2s'}{q} \geq p_{ST}$} \\
 \tau^{\frac{\rho}{2} \big(\frac{1}{s} - \frac{1}{q} \big)} & \mbox{if $2 < \frac{2s'}{q} \leq p_{ST}$}
\end{array} \right. .
\end{align*}
To conclude we write as done above that if $2<q\leqslant \frac{2d}{d-1},$ then
\begin{align*}
\Vert \mathcal{B}_{\tau,k_0,k_1}(D) \Vert_{L^s \rightarrow L^p} & \lesssim \sum_{k \geqslant k_0}^{k_1} \frac{2^{k-1}}{\sqrt{\tau}} \big \Vert \Psi(2^k(D-\sqrt{\tau})) \big \Vert_{L^{s} \rightarrow L^q} \\
& \lesssim \left\{ \begin{array}{ll}
\tau^{\frac{d}{2}\big(\frac{1}{s}-\frac{1}{2} \big)-\frac{3}{4}} & \mbox{if $\frac{2s'}{q} \geq p_{ST}$} \\
\tau^{\frac{\rho}{2}\big(\frac{1}{s}-\frac{1}{q} \big)-\frac{1}{2}} & \mbox{if $2 < \frac{2s'}{q} \leq p_{ST}$}
\end{array} \right. .
\end{align*}
If $q>\frac{2d}{d-1},$ then
\begin{align*}
\Vert \mathcal{B}_{\tau,k_0,k_1}(D) \Vert_{L^s \rightarrow L^p} & \lesssim \left\{ \begin{array}{ll}
\tau^{\frac{d}{2}\big(\frac{1}{s} - \frac{1}{q}  \big)-1} & \mbox{if $\frac{2s'}{q} \geq p_{ST}$} \\
\tau^{\frac{\rho}{2}\big( \frac{1}{s}-\frac{1}{q} \big)+\frac{d}{2}\big( \frac{1}{2}-\frac{1}{q} \big)-\frac{3}{4}} & \mbox{if $2 < \frac{2s'}{q} \leq p_{ST}$}
\end{array} \right. .
\end{align*}
The lower half of Figure \ref{figure1} is obtained by duality.

\begin{remark}
This same reasoning can be applied to spectral projectors. It yields slightly different bounds than above, where we relied on a duality trick.  More precisely, we have, for the regions defined in Figure \ref{figure1}:
\begin{itemize}
\item \textit{Region I:} $\Vert P_{\Lambda} \Vert_{L^{s} \rightarrow L^q} \lesssim [(q-2)^{-1}+1][(s'/q-1)^{-1} +1]^{2/q} \Lambda^{\rho \big(\frac{1}{s} - \frac{1}{q} \big)}$
\item \textit{Region II:} $\Vert P_{\Lambda} \Vert_{L^{s} \rightarrow L^q} \lesssim [(s'/q-1)^{-1} +1]^{2/q} \Lambda^{\rho \big(\frac{1}{s} - \frac{1}{q} \big)+d \big( \frac{1}{2} - \frac{1}{q} \big)-\frac{1}{2}}$
\item \textit{Region III:} $\Vert P_{\Lambda} \Vert_{L^{s} \rightarrow L^q} \lesssim \Lambda^{d \big(\frac{1}{s} - \frac{1}{q} \big)-1}$
\item \textit{Region IV:} $\Vert P_{\Lambda} \Vert_{L^{s} \rightarrow L^q} \lesssim [(q-2)^{-1}+1] \Lambda^{d \big(\frac{1}{s} - \frac{1}{2} \big)-\frac{1}{2}} .$
\end{itemize}
\end{remark}

\subsection{Derivative of the resolvent} Finally, we consider bounds on the derivative of the resolvent (they will be useful in the application to the electromagnetic perturbation of the Laplacian operator in Section~\ref{applications}).

\begin{theorem}[Boundedness of the derivative of the resolvent] \label{thm-Dresolvent}
Let $z:= \tau + i \varepsilon \in \mathbb{C} \setminus \lbrace 0 \rbrace$, with $\epsilon>0$. For $1\leq s \leq q \leq \infty$, for regions I, II, III, IV defined in Figure~\ref{figure2}, and for implicit constants independent of $\epsilon$,
\begin{itemize}
\item {Region $I:$} $\Vert D (D^2-z)^{-1} \Vert_{L^s \rightarrow L^q} \lesssim \tau^{\frac{\rho}{2}\big(\frac{1}{s}-\frac{1}{q} \big)},$
\item {Region $II:$} $\Vert D (D^2-z)^{-1} \Vert_{L^s \rightarrow L^q} \lesssim \tau^{\frac{\rho}{2}\big(\frac{1}{s}-\frac{1}{q} \big)+\frac{d}{2}\big(\frac{1}{2} - \frac{1}{q} \big) - \frac{1}{4}} . $
\item {Region $IV:$}  $\Vert D (D^2-z)^{-1} \Vert_{L^s \rightarrow L^q} \lesssim \tau^{\frac{d}{2}\big(\frac{1}{s} - \frac{1}{2} \big)-\frac{1}{4}}.$
\end{itemize}
In the particular case where the exponents are in duality, we have 
For $2< p \leqslant \frac{2d}{d-1},$
\begin{align*}
\Vert D (D^2-z)^{-1}  \Vert_{L^{p'} \rightarrow L^{p}} \lesssim \tau^{\frac{d-1}{2}(\frac{1}{2}-\frac{1}{p})},
\end{align*}
where the implicit constant does not depend on $\varepsilon,$ but may depend on $p.$
\end{theorem}

\begin{figure}[!]
\centering
\begin{tikzpicture}[scale=0.8]
\draw[->] (-5,0) -- (5.3,0) ;
\draw[->] (-5,0) -- (-5,10.3) ;
\draw (4.7,-0.3) node{$1/s$} ;
\draw (-5.6,10) node{$1/q$} ; 
\draw (5,0) -- (5,10) ;
\draw (-5,10) -- (5,10); 
\draw[red,dashed] (5,0)--(1.2505,3.75) ;
\draw[red,dashed] (0,5)--(-5,10) ;
\draw[blue,dashed] (0,5)--(5,5) ;
\draw[blue,dashed] (0,5)--(0,0) ;
\draw[green,dashed] (5,7.49)--(-1.66,0.833) ;
\draw (-5,4.16) node{$\bullet$} ;
\draw (-5.4,4.16) node{$\frac{d-1}{2d}$} ;
\draw[yellow,dashed] (-5,4.16)--(0,4.16) ;
\draw[yellow,dashed] (0.84,4.16)--(5,4.16) ;
\draw[purple,dashed] (5,0)--(0,7) ;
\draw[purple,dashed] (5,0)--(-2,5);
\draw[yellow,dashed] (0.84,0)--(0.84,4.16);
\draw[yellow,dashed] (0.84,5)--(0.84,7);
\draw (0,5) node{$\circ$};
\draw (0.5,4.5) node[scale=0.5]{$I$};
\draw (2,4.7) node[scale=0.5]{$IV$};
\draw (1.1,3.9) node[scale=0.5]{$II$};
\draw (0.3,3) node[scale=0.5]{$IV$};
\end{tikzpicture}
\caption{\label{figure2} Boundedness of $D(D^2-\tau-i \varepsilon)^{-1}$. The equations of the lines in this figure are as follows (using the convention that a line and its symmetric with respect to the diagonal $1/q+1/s=1$ have the same color). 
Yellow line: $\frac{1}{q} = \frac{d-1}{2d}$;  Green line: $\frac{1}{q} - \frac{1}{s} = \frac{1}{d}$; Purple line: $\frac{d-1}{d+1} \frac{1}{q} + \frac{1}{s} = 1. $}
\end{figure}
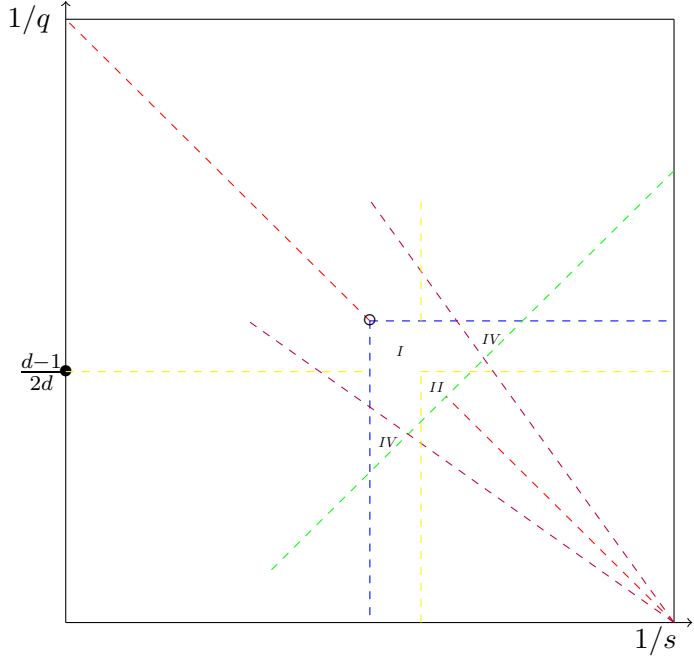

\begin{remark} Comparing the above with the statement of Theorem~\ref{thm-resolvent}, there is no region $III$ due to the green line being shifted upward. This reflects the fact that such estimates are not available in the Euclidean case.
\end{remark}

\begin{proof} The proof is similar to the above, therefore we only sketch the argument. The decomposition of the kernel is identical, and similar considerations lead to the following estimates for the non singular part of the kernel:
\begin{align*}
\Vert A^{(1)}_{\tau,k_0}  \Vert_{L^{\infty}_r} \lesssim \tau^{\rho}, \Vert \widetilde{A^{(1)}_{\tau,k_0}} \Vert_{L^{\infty}_{\lambda}} \lesssim \tau^{-\frac{1}{2}},
\end{align*}
which leads to 
\begin{align*}
\Vert \mathcal{A}_{\tau,k_0}^{(1)} \Vert_{L^{p'}\rightarrow L^p} \lesssim \tau^{\frac{d}{2}\big(1-\frac{2}{p} \big)-\frac{1}{2}}.
\end{align*}
Using the pointwise estimate
\begin{align} \label{pt-far-der}
\big \vert A^{(2)}_{\tau,k_0}(r) \big \vert \lesssim \frac{1}{(\sinh r)^{\rho} r^{\rho} }, 
\end{align}
we obtain that by the Hardy-Littlewood-Sobolev inequality $ \Vert \mathcal{A}^{(2)}_{\tau,k_0} (D) \Vert_{L^{p'}\rightarrow L^p} \lesssim 1,$ for $1/p' - 1/p \leqslant 1/d$ in the region $0<r <1.$ When $r \geqslant 1$ we rely on Lemma \ref{K-S}, and we see that a similar estimate holds. \\
To deal with the singular part, we simply write that 
\begin{align*}
D \frac{D^2-\tau}{(D^2-\tau)^2+\varepsilon^2} = \frac{(D-\sqrt{\tau})^2(D+\tau)}{(D^2-\tau)^2+\varepsilon^2} + \sqrt{\tau} \frac{D^2-\tau}{(D^2-\tau)^2+\varepsilon^2}
\end{align*}
where the first term is treated as a remainder term due to the cancellation of the singularity, and the boundedness for the second term directly follows from estimates proved above for the resolvent. 
\end{proof}

\section{Boundedness of the Fourier extension operator} 
Our aim in this section is to study the boundedness of the operator
$$
E_\lambda: L^p(\mathbb{S}^{d-1}) \to L^q(\mathbb{H}^d)
$$
for $p \neq 2$ - as we saw, the case $p=2$ is equivalent to that of the boundedness of the spectral projectors, and was fully analyzed in Section~\ref{pf-restr}.

\subsection{Isometries and Fourier transform} \label{isom}
For $x\in \mathbb{H}^d$, $\omega \in \mathbb{S}^{d-1}$, let
\begin{align*}
\langle x ,\omega \rangle = - \log [x,b(\omega)],
\end{align*}
so that
$$
h_{\lambda,\omega}(x) = e^{(\rho - i \lambda) \langle x,\omega \rangle}.
$$
We can define an action of $SO(d,1)$ on $\mathbb{S}^{d-1}$ as follows: first identify $\omega \in \mathbb{S}^{d-1}$ and $b(\omega)$ on the cone of $\mathbb{R}^{d+1}$; then, if $U \in SO(d,1)$, set 
$$
U \omega = \omega', \qquad \mbox{with} \qquad b(\omega') = \frac{1}{(Ub(\omega))_0} Ub(\omega).
$$
Using the identity 
\begin{align*}
\langle U z , U \omega \rangle = \langle z , \omega \rangle + \langle U \textbf{0} , U \omega \rangle
\end{align*}
from~\cite{Bray}, equalities (2.1)-(2.3), we see that
\begin{align*}
\widetilde{f \circ U}(\lambda,\omega) &= \int_{\mathbb{H}^d} f(Ux) e^{(-i\lambda + \rho)\langle x , \omega \rangle} dx \\
&=  \int_{\mathbb{H}^d} f(x) e^{(-i\lambda + \rho)\langle U^{-1} x , \omega \rangle} dx \\
&= \int_{\mathbb{H}^d} f(x) e^{(-i\lambda + \rho)\big(\langle x ,  U \omega \rangle + \langle U^{-1} \textbf{0} , \omega \rangle  \big)} dx \\
&=[U^{-1} \textbf{0}, b(\omega)]^{i\lambda-\rho} \widetilde{f}(\lambda,U\omega).
\end{align*}

\subsection{Lower bounds for the Fourier extension operator: proof of Proposition~\ref{proplb}}
We gather a few observations which together give the proof of this proposition. We use successively the symmetry group of $\mathbb{H}^d$ and both examples from Section \ref{lower-b}.
 
\subsubsection{Lorentz boosts}
Consider
$U = \left( \begin{array}{lll}
\cosh t & \sinh t & 0 \\
\sinh t & \cosh t & 0 \\
0 & 0 & I_{d-1}
\end{array} \right) .
$
We can parametrize the sphere $\mathbb{S}^{d-1} \subset \mathbb{R}^d$ by $(\theta_1,\sigma) \in [0,\pi] \times \mathbb{S}^{d-2}$, and the formula $\omega = (\cos \theta_1, \sin \theta_1 \sigma)$. A small computation shows that
$$
U \omega = \frac{1}{\cosh t + \sinh t \cos \theta_1} 
\begin{pmatrix}
\sinh t + \cosh t \cos \theta_1 \\ \sin \theta_1 \sigma 
\end{pmatrix}.
$$
Introducing the new variable $\theta_1 '$ defined as
$$
\begin{pmatrix}
\cos \theta'_1 \\ \sin \theta'_1
\end{pmatrix} = \frac{1}{\cosh t + \sinh t \cos \theta_1} 
\begin{pmatrix}
\sinh t + \cosh t \cos \theta_1 \\ \sin \theta_1
\end{pmatrix},
$$
let
\begin{align*}
\omega' = U \omega =  \begin{pmatrix}
 \cos \theta_1' \\ \sin \theta_1' \sigma 
\end{pmatrix}.
\end{align*}
An elementary computation shows that
$$
\frac{d\theta'_1}{d\theta_1} = \frac{1}{\cosh t + \sinh t \cos \theta_1} = [U^{-1} \mathbf{0},b(\omega)]^{-1}.
$$
As a consequence, the volume element on the sphere
becomes 
$$
d\omega = (\sin \theta_1)^{d-2} d\theta_1\,d\sigma = [U^{-1} \textbf{0},b(\omega)]^{d-1} (\sin \theta_1')^{d-2} d\theta_1' d\sigma = [U^{-1} \textbf{0},b(\omega)]^{d-1}  d\omega'.
$$
We also note that
$$
\cos \theta_1 = \frac{\sinh t - \cos \theta'_1 \cosh t}{\sinh t \cos \theta'_1 - \cosh t},
$$
which leads to the equivalent, as $t \to \infty$,
$$
[U^{-1} \mathbf{0},b(\omega)] = \cosh t + \sinh t \cos \theta_1 \sim e^{-t} \frac{2}{1 - \cos \theta'_1} \to 0
$$
except for $\theta'_1 = 0$.
\\
We learn from  Section \ref{isom} that
$$
(E_\lambda g) \circ U = E_\lambda([U^{-1} \textbf{0}, b(\omega)]^{i\lambda-\rho}  g \circ U).
$$
Therefore, 
$$
\| E_\lambda \|_{L^p \to L^q} \geq \frac{\| (E_\lambda g) \circ U \|_{L^q(\mathbb{H}^{d})} } 
{ \| [U^{-1} \textbf{0}, b(\omega)]^{-\rho} g \circ U \|_{L^p(\mathbb{S}^{d-1})}}.
$$
Since $U$ is an isometry, the numerator is independent of $\| U \|$. As for the denominator we can write, changing variables from $\omega$ to $\omega'$ that
\begin{align*}
\|[U^{-1} \textbf{0}, b(\omega)]^{-\rho} g \circ U \|_{L^p(\mathbb{S}^{d-1})}^p &= \int_{\mathbb{S}^{d-1}}  [U^{-1} \textbf{0}, b(\omega)]^{-\rho p} |g (U \omega)|^p \,d\omega \\
& = \int_{\mathbb{S}^{d-1}} [U^{-1} \textbf{0}, b(\omega)]^{2 \rho-\rho p} \vert g (\omega') \vert^p \,d\omega'.
\end{align*}
But as $t \to \infty$, $[U^{-1} \textbf{0}, b(\omega)] = \cosh t + \sinh t \cos \theta \to 0$ almost everywhere in $\omega'$, which implies that
$$
\|[U^{-1} \textbf{0}, b(\omega)]^{-\rho} g \circ U \|_{L^p(\mathbb{S}^{d-1})} \to 0
$$
if $p<2$.
This shows that $L^p \to L^q$ boundedness requires $p \geq 2$.
\subsubsection{The radial example} Testing $E_\lambda$ on the test function $1$ implies that
$$
\| E_\lambda \|_{L^p \to L^q} \geq \frac{| \textbf{c}(\lambda)|^{-1} \| \Phi_\lambda \|_{L^q(\mathbb{H}^{d})}} {\| 1 \|_{L^p(\mathbb{S}^{d-1})}} \gtrsim \left\{
\begin{array}{ll} \infty & \mbox{if $q\leq 2$} \\
1 & \mbox{if $2 < q \leq \frac{2d}{d-1}$} \\
\lambda^{\rho-\frac d q} & \mbox{if $ q > \frac{2d}{d-1}$}
\end{array}
\right. .
$$
\subsubsection{The Knapp example} Testing $E_\lambda$ on the test function $\varphi_\delta$ implies that, if $\delta < \lambda^{-1/2}$,
$$
\| E_\lambda \|_{L^p \to L^q} \geq \frac{\| E_{\lambda} \varphi_\delta \|_{L^q(\mathbb{H}^{d})}} {\| \varphi_\delta \|_{L^p(\mathbb{S}^{d-1})}} \gtrsim (q-2)^{-1/2} \lambda^{\frac{\rho}{2} - \frac{\rho}{q}} \delta^{\rho - \frac{d-1}{p}}.
$$
The right-hand side is maximal for $\delta = \lambda^{-1/2}$, which gives
$$
\| E_\lambda \|_{L^p \to L^q} \gtrsim (q-2)^{-1/2}  \lambda^{\frac \rho p - \frac \rho q}.
$$

\section{Applications} \label{applications}
\subsection{Boundedness for small frequencies}
For our applications we will need a version of Theorems \ref{thm-projector} and \ref{thm-resolvent} for $0<\Lambda, \tau\leqslant 1.$ 
\begin{lemma} \label{small-freq}
Let $p > 2,  \Lambda>0, $ Then
\begin{align*}
\Vert P_{\Lambda} \Vert_{L^{p'}\rightarrow L^p} \lesssim \Lambda^2.
\end{align*}
Moreover if $p< \infty$ and $1-\frac{2}{p} \leqslant \frac{2}{d}$ then
\begin{align*}
\sup_{\varepsilon>0, \tau \in [0,1]} \Vert (D^2 - \tau - i \varepsilon)^{-1} \Vert_{L^{p'}\rightarrow L^p} \lesssim 1.
\end{align*}
Finally if $1-\frac{2}{p} \leqslant \frac{1}{d}$ then
\begin{align*}
\sup_{\varepsilon>0, \tau \in [0,1]} \Vert D(D^2 - \tau - i \varepsilon)^{-1} \Vert_{L^{p'} \rightarrow L^p} \lesssim 1.
\end{align*}
\end{lemma}
\begin{proof}
For the first estimate, we see that the kernel is bounded pointwise by $\Lambda^2$ in the region $0<s\leqslant 1$ by a power series expansion in \eqref{albatros1}, \eqref{albatros2}. In the region $s \geqslant 1$ the kernel is bounded pointwise by $e^{-\rho s} \Lambda s /\langle \Lambda s \rangle, $ and therefore we can use Lemma \ref{K-S} to  conclude. \\
For the resolvent, we can handle the region far from the singularity ($\lambda \geqslant 2 \sqrt{\tau}$) by appealing to the pointwise bound \eqref{point-far-3} with $\beta = \rho-1$ if $d \geqslant 3$ and the log bound in \eqref{point-far-2} if $d=2.$ For the region that is close to the singularity, we use the fact that the kernel is bounded pointwise by $\min \{ s^{2-d}, e^{-\rho s} \}.$ Then we appeal to the Hardy-Littlewood inequality for the region $s \leqslant 1$ and Lemma \ref{K-S} when $s >1.$ \\
The boundedness of the derivative of the resolvent is proved similarly, relying on the pointwise bound \eqref{pt-far-der} far from the singularity and $\sqrt{\tau} \min \{ s^{2-d}, e^{-\rho s} \}$ close to it.
\end{proof}
\subsection{Smoothing estimates}
We start with a smoothing effect for the homogenenous Schr\"{o}dinger equation on $\mathbb{H}^d.$ 
Recall that on $\mathbb{H}^d$ we can define fractional derivatives through functional calculus as 
\begin{align*}
D^{\gamma} := \int \lambda^{\gamma} P_{\lambda} \, d\lambda.
\end{align*}
With this definition we can state our results.
\begin{theorem} \label{smoothing-NLS-hom}
Let $2<p<\infty$ be such that $1/p'-1/p \leqslant 2/d$ if $d \geqslant 3.$ Let 
\begin{align*}
\gamma_p = \left\{ \begin{array}{ll}
\displaystyle 1-d  \big(\frac{1}{2} -\frac{1}{p} \big) & \textrm{if} \ p>p_{ST} \\
\displaystyle \frac{1}{2} - \frac{d-1}{2} \big(\frac{1}{2} - \frac{1}{p} \big) & \textrm{if} \ 2< p \leqslant p_{ST}
\end{array} \right. .
\end{align*}
We have
\begin{align*}
\Vert D_x ^{\gamma_p}  e^{it\Delta_{\mathbb{H}^d}} f \Vert_{L^p_x(L^2_t(\mathbb{R}))} \lesssim \Vert f \Vert_{L^2_x}.
\end{align*}
Moreover 
\begin{align*}
\bigg \Vert \int_0 ^T D^{\gamma_{p}}_x e^{it\Delta_{\mathbb{H}^d}} f(t) dt \bigg \Vert_{L^2_x(\mathbb{H}^d)} \lesssim \Vert f \Vert_{L^{p'}_x(L^2_t (\mathbb{R}))}
\end{align*}
and
\begin{align*}
\bigg \Vert D_x ^{2 \gamma_{p}} \int_{-\infty}^{+\infty} e^{i(t-s)\Delta_{\mathbb{H}^d}} f(s) ds \bigg \Vert_{L^{p}_x(L^2_t (\mathbb{R}))} \lesssim \Vert f \Vert_{L^{p'}_x (L^2_t(\mathbb{R}))}.
\end{align*}
\end{theorem}
\begin{remark}
Note that under the condition $1/p'-1/p \leqslant 2/d, $ if $d \geqslant 3$ or $d=2,$ we have $\gamma_p \geqslant 0.$
\end{remark}
\begin{proof}
This is a corollary of the restriction estimate from Theorem \ref{thm-projector}. Following \cite{RV} (proofs of Propositions 2.1, 2.2 and 2.3), we change variables, use the Plancherel theorem in $t$ and then Minkowski's inequality in $x$ to obtain
\begin{align*}
\bigg \Vert \int_0 ^{\infty} e^{it\lambda^2}  \lambda^{ \gamma_p} \big[P_{\lambda} f \big] d\lambda \bigg \Vert_{L^p_x L^2_t} &= \frac{1}{2} \bigg \Vert \int_0 ^{\infty} e^{it\lambda} \lambda^{ \frac{\gamma_p}{2}} \big[P_{\sqrt{\lambda}} f \big]  \lambda^{-1/2} d\lambda \bigg \Vert_{L^p_x L^2_t} \\
&  \lesssim \big \Vert \lambda^{ \frac{\gamma_p-1}{2}}  P_{\sqrt{\lambda}} f \big \Vert_{L^p_x L^2 _{\lambda}}  \lesssim \big \Vert  \lambda^{ \frac{\gamma_p-1}{2}} P_{\sqrt{\lambda}} f \big \Vert_{L^2 _{\lambda} L^p_x}. 
\end{align*}
Next we notice that a straightforward consequence of Theorem \ref{thm-projector} and Lemma \ref{small-freq} is that $\Vert E_{\sqrt{\lambda}} \Vert_{L^2_{\omega} \rightarrow L^p} \lesssim \min \big\{ \lambda^{\frac{1}{2}},  \lambda^{- \frac{2\gamma_p-1}{4}} \big\} .$ Writing $P_{\sqrt{\lambda}} = E_{\sqrt{\lambda}} R_{\sqrt{\lambda}},$ we can therefore bound the above by a constant times
\begin{align*}
\bigg(\int_0 ^{\infty}  \Vert R_{\sqrt{\lambda}} f \Vert_{L^2_{\omega}}^2 \lambda^{-1/2} d\lambda \bigg)^{1/2} & \lesssim \bigg(\int_0 ^{\infty}  \Vert \widetilde{f}(\lambda,\omega) \Vert_{L^2_{\omega}}^2 \vert \textbf{c} (\lambda) \vert^{-2} d\lambda \bigg)^{1/2} \lesssim \Vert f \Vert_{L^2_x}.
\end{align*}
For the first inequality we changed variables, and the last line comes from Plancherel's Theorem for the Fourier Helgason transform. \\
The second inequality is the dual of the estimate just obtained. \\
The third inequality of the theorem is proved using successively the first then second inequality of the theorem. 
\end{proof}
\begin{remark} \label{low-mult}
Note that the proof also implies that the same result holds with the multiplier $\langle D_x \rangle.$
\end{remark}
\noindent In the inhomogeneous case we have
\begin{theorem} \label{smoothing-NLS-inhom}
Let $u$ solve
\begin{equation*}
\begin{cases}
i \partial_t u + \Delta_{\mathbb{H}^d} u = f \\
u(x,0) = 0 
\end{cases} .
\end{equation*}
Then for $2 < p < \infty$ such that $$\frac{1}{p'}-\frac{1}{p} \leqslant \frac{2}{d} \ \ \ \textrm{if} \ \  d \geqslant 3,$$
we have 
\begin{align*}
\Vert u \Vert_{L^{p}_x(L^2_t(\mathbb{R}))} \lesssim \Vert f \Vert_{L^{p'}_x(L^2_t(\mathbb{R}))} .
\end{align*}
\end{theorem}
\begin{proof}
The proof is the same as in \cite{RV} (proof of Proposition 2.5) using Theorems \ref{thm-resolvent}, \ref{small-freq} and the remark \ref{low-mult} above.
\end{proof}
\begin{remark}
In Theorems \ref{smoothing-NLS-hom} and \ref{smoothing-NLS-inhom} we obtain a wider range of exponents than in the euclidean case. 
\end{remark}

\subsection{Limiting absorption principle for electromagnetic Schr\"{o}dinger on $\mathbb{H}^d$}
The resolvent estimates obtained in Theorem \ref{thm-resolvent} classically lead to a limiting absorption principle for small electromagnetic potentials.  \\
Starting with the electric case, we have
\begin{theorem}
Let $2<p<\infty$ be such that $1/p'-1/p \leqslant 2/d$ if $d \geqslant 3.$ Let $V \in L^{\frac{p}{p-2}}_x.$ Then there exists $\delta>0$ such that if $\Vert V \Vert_{L^{\frac{p}{p-2}}_x} < \delta, $ then denoting $z:=\tau + i \varepsilon, \varepsilon >0,$ we have
\begin{align*}
\sup_{\tau, \varepsilon>0} \Vert (D^2 + V - z)^{-1} \Vert_{L^{p'}\rightarrow L^p} \lesssim 1 .
\end{align*} 
Moreover the limit of $(D^2 + V - z)^{-1}$ as $\varepsilon \to 0$, denoted $(D^2 + V -\tau-i0)^{-1}, $ exists in the sense of distributions and satisfies the same bound as above.
\end{theorem}
\begin{proof}
We first note that for $f \in L^p_x,$, using Theorems \ref{thm-resolvent}, \ref{small-freq}, and then H\"{o}lder's inequality, we have
\begin{align*}
\Vert (D^2-z)^{-1} (V f) \Vert_{L^{p}_x} \lesssim \Vert V f \Vert_{L^{p'}_x} \lesssim \Vert V \Vert_{L^{\frac{p}{p-2}}_x} \Vert f \Vert_{L^p_x}.
\end{align*}
We conclude that
\begin{align} \label{Neumann}
\Vert (D^2-z)^{-1} V \Vert_{L^{p} \rightarrow L^p} \lesssim \Vert V \Vert_{L^{\frac{p}{p-2}}_x}.
\end{align}
Here $(D^2-z)^{-1} V$ denotes the composition of the the flat resolvent with the multiplication by $V$ operator.
We stress that the implicit constants in \eqref{Neumann} are independent of $\varepsilon$ and $\tau$ since under the condition $1/p'-1/p \leqslant 2/d$ if $d \geqslant 3$ or $d=2,$ the exponent of $\tau$ in Theorem \ref{thm-resolvent} is negative. \\
Therefore by a Neumann series argument, the operator $\big( I + (D-z)^{-1} V \big)$ is invertible in $L^p$ if $\Vert V \Vert_{L^{\frac{p}{p-2}}_x}$ is small enough. The desired estimates directly follow from the resolvent identity
\begin{align} \label{resolvent-id}
(D^2+V-z)^{-1} = \big(I + (D^2-z)^{-1} V\big)^{-1} (D^2-z)^{-1}.
\end{align}
To prove the second assertion we note that the limit exists in a distributional sense when $V=0,$ as can be seen from the explicit formula of the kernel of the resolvent. Then using the estimate just proved, along with \eqref{resolvent-id} and Fatou's lemma, the result directly follows.
\end{proof}
In the magnetic case, we have a similar theorem, although we lose uniformity on $\tau$ in the bounds above. 
\begin{theorem}
Let $p>2$ be such that $1/p'-1/p \leqslant 1/d.$ Let $A \in \big(L^{\frac{p}{p-1}}_x\big)^d$ be such that $ \nabla \cdot A \in L^{\frac{p}{p-2}}_x  .$ Let $1>\tau_0>0.$ Then there exists $\delta(\tau_0)>0$ such that if $\Vert A \Vert_{\big(L^{\frac{p}{p-2}}_x\big)^d} ,  \Vert \nabla \cdot A \Vert_{L^{\frac{p}{p-2}}_x} < \delta,$ then denoting $z:=\tau + i \varepsilon, \varepsilon >0,$ we have
\begin{align*}
\sup_{\varepsilon>0,\tau \in (0,\tau_0^{-1}]} \Vert (D^2 + i \big(A \cdot \nabla + \nabla \cdot A \big) - z)^{-1} \Vert_{L^{p'}\rightarrow L^p} \lesssim 1 .
\end{align*} 
Moreover the limit of $(D^2 + i \big(A \cdot \nabla + \nabla \cdot A \big) - z)^{-1}$ as $\varepsilon \to 0$, denoted $(D^2 + i \big(A \cdot \nabla + \nabla \cdot A \big) -\tau-i0)^{-1}, $ exists in the sense of distributions and satisfies the same bound as above.
\end{theorem}
\begin{proof}
The proof is very similar to the electric case, replacing \eqref{Neumann} with 
\begin{align*}
\big \Vert (D^2 -z)^{-1} \nabla \cdot A \big \Vert_{L^p \rightarrow L^p} \lesssim \Vert \nabla \cdot A \Vert_{L^{\frac{p}{p-2}}_x}  ,  \Vert (D^2 -z)^{-1} A \cdot \nabla \Vert_{L^p \rightarrow L^p} \lesssim_{\tau_0} \Vert A \Vert_{\big(L^{\frac{p}{p-1}}_x\big)^d} ,
\end{align*}
where the second inequality is proved by duality. Moreover the implicit constant does not depend on $\varepsilon,$ but depends on $\tau_0$ in the second inequality.
\end{proof}

\end{document}